\documentclass[11pt]{article}
\usepackage{bbm}
\usepackage{amsmath,amsthm,latexsym,amssymb}
\usepackage{enumerate}
\usepackage{algorithm,enumerate}
\usepackage{algpseudocode}
\usepackage{hyperref}
\allowdisplaybreaks

\usepackage{color}\usepackage{graphicx}

\def\cS{{\mathcal S}}

\newcommand{\set}[1]{\left\{#1\right\}}

\def\cH{\mathcal{H}}

\def\cP{\mathcal{P}}

\def\ii_(#1,#2){i_{#1}^{#2}}

\def\bx{{\bf x}}

\def\a{\alpha}

\def\d{\delta}
\def\D{\Delta}

\def\z{\zeta}

\def\l{\lambda}

\def\s{\sigma}

\def\cD{\mathcal{D}}
\def\cE{\mathcal{E}}

\def\cS{\mathcal{S}}
\def\cT{\mathcal{T}}

\def\E{{\bf E}}

\renewcommand{\Pr}{\operatorname{\bf Pr}}
\newcommand\bfrac[2]{\left(\frac{#1}{#2}\right)}

\def\bx{{\bf x}}

\def\kG{{\sc k-MatchTINF }}

\newtheorem{theorem}{Theorem}[section]

\newtheorem{lemma}[theorem]{Lemma}

\newtheorem{remthm}[theorem]{Remark}
\newtheorem{claim}[theorem]{Claim}

\newtheorem{notation}[theorem]{Notation}

\title{On a $k$-matching algorithm and finding $k$-factors in random graphs with minimum degree $k+1$ in linear time}
\author{Michael Anastos\footnote{
Institut f\"ur Mathematik, Freie Universit\"at Berlin, 14195 Berlin, Germany.
E-mail: manastos@zedat.fu-berlin.de.
}
}
\begin{document}

\maketitle
\begin{abstract}
We prove that for $k+1\geq 3$ and $c>(k+1)/2$ w.h.p. the random graph on $n$ vertices, $cn$ edges  and minimum degree $k+1$ contains a (near) perfect $k$-matching. As an immediate consequence we get that w.h.p. the $(k+1)$-core of $G_{n,p}$, if non empty,  spans a (near) spanning $k$-regular subgraph. This improves upon a result of Chan and Molloy \cite{CM} and completely resolves a conjecture of  Bollob\'as, Kim and  Verstra\"{e}te  \cite{BKV}. In addition, we show that w.h.p. such a subgraph can be found in linear time.  A substantial element of the proof is the analysis of a randomized algorithm for finding $k$-matchings in random graphs with minimum degree $k+1$. 
\end{abstract}
\section{Introduction}
The main focus of this paper is finding a (near) perfect $k$-matching in  random graphs with minimum degree $k+1$.
For $k,n,m\in \mathbb{N}$ we let $G_{n,m}^{\delta \geq k+1}$ be the random graph chosen uniformly at random from the set of all graphs on $n$ vertices, with $m$ edges and minimum degree $k+1$. Given a graph $G$, a $k$-matching of $G$ is a set of edges $M\subseteq E(G)$ such that each vertex in $V(G)$ is incident to at most $k$ edges in $M$ (as opposed to a matching of size $k$). Such a matching is called (near) perfect if it is of size $\lfloor kn/2\rfloor$. 

When $m=(k+1)n/2$ then $G_{n,m}^{\delta \geq k+1}$ is distributed as a random $(k+1)$-regular graph. In this case, when $n$ is even Robinson and Wormald \cite{RW} showed, via the small cycle conditioning method, that w.h.p.\footnote{We say that a sequence of events $\{\mathcal{E}_n\}_{n\geq 1}$
holds {\em{with high probability}} (w.h.p.\@ in short) if $\lim_{n \to \infty}\Pr(\mathcal{E}_n)=1-o(1)$.} $G_{n,m}^{\delta \geq k+1}$ spans $k+1$ pairwise edge disjoint perfect matchings. However  this result cannot be extended to slightly larger values of $m$ as it fails to be true due to the appearance of vertices whose neighborhood contains $k+2$ vertices of degree $k+1$. 

The problem of finding a large $k$-matching in $G_{n,m}^{\delta \geq k+1}$ has been studied in the case $k=2$ in \cite{AF}, \cite{F}.

A graph that is known to be distributed as $G_{n',m'}^{\delta\geq k+1}$ for suitable $n',m'$ is the $(k+1)$-core of the random graph $G_{n,p}$. We denote by $G_{n,p}$  the random graph on $n$ vertices where each edge appears independently with probability $p$. The $k$-core of a graph $G$, denoted by $G^{(k)}$, is the unique maximal subgraph of $G$ of minimum degree $k$. {\L}uczak \cite{l} showed that w.h.p. the order of $G_{n,p}^{(k+1)}$ is either  zero or linear in $n$. Later, Pittel, Spencer and Wormald \cite{PSW} determined the threshold for the appearance of a non-empty $k$-core of $G_{n,p}$ to be $p=c_k/n$ where $c_k=k+\sqrt{k\log k}+ o(\sqrt{k})$. 

In \cite{BKV}, Bollob\'as, Kim and Verstra\"{e}te studied the existence of regular subgraphs in $G_{n,p}$. Let $\psi_k/n$ be the threshold for appearance of a $k$-regular subgraph of $G_{n,p}$. De facto $\psi_k\geq c_k$.
They proved that $\psi_k \leq \gamma_k$ where $\gamma_k=4k+o(k)$.
In addition, they showed that $\psi_k$ is strictly larger than $c_k$ for $k=3$ and conjectured that this statement is true for all $k\geq 3$. On the other hand Pretti and Weigt \cite{PW},  using techniques from statistical physics, predicted the contrary, that is $c_k=\psi_k$ for $k\geq 4$.
Mitsche,  Molloy and  Pralat \cite{MMP} proved that for sufficiently large $k$, $\psi_k-c_k\leq e^{-k/300}$ while Gao \cite{Gao} proved that if $|p-\psi_k/n|$ is sufficiently small then any $k$-regular subgraph of $G_{n,p}$
must contain all but at most $\epsilon_k n$ vertices of the $k$-core of $G_{n,p}$ for some $\epsilon_k$ that tends to 0 as  $k\to \infty$.

Bollob\'as, Kim and Verstra\"{e}te also conjectured that if
$c>c_{k+1}$  then w.h.p. $G_{n,c/n}^{(k+1)}$ spans a $k$-regular subgraph. Chan and Molloy proved that this conjecture is true for all sufficiently large $k$ by showing that $G_{n,c/n}^{(k+1)}$, minus any vertex when $kn$ is odd, spans a $k$-factor for $c>c_k$ \cite{CM}.  We completely resolve this conjecture by proving that this conjecture is true for all values of $k\geq 2$.

A $k$-factor of $G$ is a $k$-regular spanning subgraph of $G$.  We say that $G$ is $k$-factor critical if $G\setminus \{v\}$ spans a $k$-factor for every $v\in V(G)$. We let $\cP_k$ be the set of graphs $G$ that either have a $k$-factor or, in the case that $k|V(G)|$ is odd, are $k$-factor critical.
  
The main result of this paper is the following Theorem.
\begin{theorem}\label{thm:matchings}
Let $3\leq k+1=O(1)$ and { $(k+1)/2<c=O(1)$}. Then w.h.p. $G_{n,cn}^{\delta \geq k+1} \in \cP_k$.
In addition, w.h.p. we can find the corresponding $k$-factor in $O(n)$ time.
\end{theorem}
If non empty, $G_{n,p}^{(k+1)}$ is distributed as $G_{n',m'}^{\delta \geq k+1}$ for some $n',m'$ that satisfy $n'=\Omega(n)$ and $m' \leq |E(G_{n,p})|=\Omega(n)$. Hence as an immediate corollary we have that for $3\leq k=O(1)$ and $0\leq c =O(1)$ w.h.p. either  $G_{n,c/n}^{(k+1)}=\emptyset$ or  $G_{n,c/n}^{(k+1)}\in \cP_k$. In this paper we prove the slightly stronger, hitting time analogue of this corollary.

Denote by $F_0,F_1,....,F_N$, $N=\binom{n}{2}$ the random graph process. That is, $F_0$ is the empty graph on $n$ vertices and $F_i$ is formed by adding to $F_{i-1}$ an edge chosen uniformly at random from the ones that are not present in $F_{i-1}$ for $i\in [n(n-1)/2]$. Let $\sigma_{k+1}=\min\{i:F_i^{(k+1)} \neq \emptyset\}.$ 

\begin{theorem}\label{cor:matchings}
Let $3\leq k+1=O(1)$. Then w.h.p. $F_i^{(k+1)} \in \cP_k$ for $i\geq \sigma_{k+1}$. 
\end{theorem}
Our proof approach differs from the one used in \cite{CM} which is driven by verifying a generalized Tutte-like condition. It relies on the application of a randomized algorithm for finding a large $k$-matching, which we call \kG\@ ($k$-Match Those In Need First). 
\kG goes as follows: Its input is a graph $G$. As it progresses it grows a $k$-matching $M$ while in parallel it removes edges from $G$. At each step it keeps track of a set $Z$ that consists of vertices that are incident to $\ell$ edges in $M$ but to at most $k-\ell$ edges in $G$. Think of these as the {\em{dangerous}} vertices. If dangerous vertices exist then \kG chooses a random edge $e \in E(G)$ that is incident to a dangerous vertex, otherwise it chooses 
a random edge $e \in E(G)$ incident to a vertex $v$ whose $M$-degree is minimized.
Then, it adds $e$ to  $M$, removes it from $G$ and deletes the edges incident to any vertex in $G$ which is incident to $k$ edges in $M$. 

We will prove that \kG outputs a matching $M$ of size $kn/2-o(n)$. We will then augment $M$ into a (near) perfect $k$-matching using alternating paths in $O(n)$ time.

We split the rest of the paper as follows: In Section \ref{section:model} we describe the model that we will use to analyze \kG. A full description of \kG is given in Section \ref{section:alg} while the main part of  its analysis is given in Section \ref{sec:analysis}.  In Section \ref{sec:reserve} we show  that one may reserve some edges before applying \kG to $G_{n,m}^{\delta \geq k+1}$ and still get a $k$-matching of size $kn/2-o(n)$. We then use those edges to augment the outputted $k$-matching into a k-factor. This last step takes place  in Section \ref{section:matchings}.

\section{Random Sequence Model}\label{section:model}
To analyse $G_{n,m}^{\delta \geq k+1}$ we use a variation of Bollob\'{a}s configuration model \cite{BolCM} which we refer to as the random sequence model. Given $n,m \in \mathbb{N}$ and a sequence of size $2m$, $\bx=(x_1,x_2,...,x_{2m})\in [n]^{2m}$ we define the multigraph $G_\bx$ by $V(G_\bx):=[n]$, $E(G_\bx):=\{\{x_{2j-1},x_{2j}\}:j\in [m]\}$. Thus $G_\bx$ is a graph on $n$ vertices with  $m$ edges. The degree of some vertex $v\in [n]$ with respect to the sequence $\bx$ is equal to the number of times it appears in $\bx$, i.e. $d_\bx(v)=|\{i: x_i=v, 1\leq i\leq 2m\}|$. We let $\cS_{n,2m}^{\delta \geq k+1}$ be the set of  sequences $\bx=(x_1,x_2,...,x_{2m})$ such that $d_{\bx}(i)\geq k+1$ for $i \in [n]$.
If $\bx$ is chosen uniformly at random from $\cS_{n,2m}^{\d \geq k+1}$ then  $G_\bx$ is close in distribution to $G_{n,m}^{\d \geq k+1}$. Indeed, conditioned on $G_\bx$ being simple, the distributions of $G_\bx$ and $G_{n,m}^{\d \geq k+1}$ are identical. Both are uniform over the simple graphs on $n$ vertices with $m$ edges and minimum degree $k$. Each such graph will correspond to $m!2^m$ sequences in $\cS_{n,2m}^{\d \geq k+1}$.

Similarly to \cite{AF}, \cite{F}, the model that we will use to analyze  \kG will be slightly more complicated than the one above.

For $i \in \mathbb{N}$, a list $L$ of size $i$ is a map from $\{0,1,...,i-1\}$ to $\mathbb{Z}_{\geq 0}$, i.e. $L:\{0,1,...,i-1\} \mapsto \mathbb{Z}_{\geq 0}$.

Given $n,m \in \mathbb{N}$, a list  $L^i$ of size $i+2$ for  $0\leq i\leq k$ such that $\sum_{i=0}^k\sum_{j=0}^{i+1}L^{i}(j)=n$ and  a map $y:[n]\mapsto \{0,1,...,k\} \times \{0,1,...,k,k+1\}$ that maps $L^i(j)$ many elements to $(i,j)$ we let $\bigg(\cS_{n,2m}^{L^0,L^1,...,L^k},y(\cdot)\bigg)$  be the set of sequences in $[n]^{2m}$ such that for $i\in [n]$ and $\bx \in \bigg(\cS_{n,2m}^{L^0,L^1,...,L^k},y(\cdot)\bigg)$
\begin{itemize}
\vspace{-3mm} \item[(i)] if $y_2(i)\leq y_1(i)$ then the element $i$ appears exactly $y_2(i)$ times in  \bx,
\vspace{-3mm} \item[(ii)] if $y_2(i) = y_1(i)+1$ then the element $i$ appears at least $y_2(i)$ times in  \bx.
\end{itemize} 

Hence  when $L^i(j) \neq 0$ iff $i=k, j=k+1$  we have,
$$\bigg(\cS_{n,2m}^{L^0,L^1,...,L^k},y(\cdot)\bigg)=\cS_{n,2m}^{\delta \geq k+1}.$$

\kG\ (presented later in Section 3) has as an input a graph $G=G_{n,m}^{\d\geq k+1}$ which we simulate using a sequence in $\cS_{n,2m}^{\delta \geq k+1}$. As the algorithm progresses edges are revealed and vertices are matched. To capture the amount of information revealed, or equivalently the current distribution of the graph, we will have to keep track of quantities like ``exactly $L^{k-i}(k-i+1)$ vertices have been matched exactly $i$ times but still have degree at least $k-i+1$" or 
``exactly $L^{k-i}(j)$ vertices have been matched exactly $i$ times and have degree $j$" for $0\leq j\leq k-i$".

Let 
\begin{equation}\label{eq:fk}
    f_k(\lambda)=e^{\lambda}-\sum_{i=0}^{k-1} \frac{\lambda^i}{i!}.
\end{equation} 
The next two Lemmas describe a typical element of $\bigg(\cS_{n,2m}^{L^0,L^1,...,L^k}, y(\cdot)\bigg)$. Let $\bx$ be an element of $\bigg(\cS_{n,2m}^{L^0,L^1,...,L^k}, y(\cdot)\bigg)$ chosen  uniformly at random. We know that the vertices (integers) that are mapped by $y(\cdot)$ to $(i_1,i_2)$ with $i_2\leq i_1$ appear exactly $i_2$ times in $\bx$. Let $d_1,d_2,...,d_{n'}$ be the degrees of the rest of the vertices in $\bx$. Lemma \ref{lem:equiv} states that the joint distribution of $d_1,d_2,...,d_{n'}$ is the same as the joint distribution of $\cP_1,\cP_2,...,\cP_{n'}$ where $\cP_j$ is a $Poisson(\lambda)$ random variable conditioned on being at least $y_1(j)+1$ for $j\in[n']$, for some universal, well tune, value of $\lambda$. Thereafter in Lemma \ref{lem:degDistributions} we prove concentration of the number of vertices that are mapped by $y(\cdot)$ to $(i,i+1)$ and have degree $j$  for $0\leq i\leq k$ and $i+1\leq j\leq \log^2 n.$

\begin{lemma}\label{lem:equiv}
Let $n,m \in \mathbb{N}$, $L^i$ be a list of size $i+1$ for $0\leq i\leq k$ such that $\sum_{i=0}^k\sum_{j=0}^{i+1}L^{i}(j)=n$. Also let  $y:[n]\mapsto \{0,1,...,k\} \times \{0,1,...,k,k+1\}$ be such that $|y^{-1}(i,j)|=L^i(j)$ for $0\leq i \leq k$ and $0\leq j\leq k+1$. 
Let $\bx$ be chosen uniformly at random from $\bigg(\cS_{n,2m}^{L^0,L^1,...,L^k}, y(\cdot)\bigg)$. Let $R= \sum_{i=0}^k\sum_{j=0}^{i}jL^i(j)$  and assume that 
\begin{equation}\label{cond}
    2m-R \geq \sum_{i=0}^k (i+1) L^i(i+1).
\end{equation}
Let $\lambda$ be the unique positive real number that satisfies 
\begin{equation}\label{eq:lambda}
\sum_{i=0}^{k}\frac{{\l}f_{i-1}({\l})}{f_i({\l})}L^i(i+1)=2m-R.
\end{equation}
For $1\leq i \leq k$ let $\cP_{\geq i}$ denote the {\em truncated at $i$ Poisson($\lambda$}) random variable, i.e
$$\Pr(\cP_{\geq i}=t)= \frac{e^{-\lambda}\l^t/t!}{1-\sum_{j=1}^{k-1}e^{-\l}\l^j/j!}
=\frac{{\l}^t}{t!f_i({\l})},\hspace{1in}\text{ for }t\geq i.$$
Let $\{Z_j:j \in [n]\}$  be a set of independent random variables such that,
\begin{itemize}
\vspace{-3mm} \item[(i)] $Z_j$ is distributed as $\cP_{ \geq y_2(j)}$ if $y_2(j)=y_1(j)+1$,
\vspace{-3mm} \item[(ii)] $Z_j=y_2(j)$ if  $y_2(j)\leq y_1(j)$.
\end{itemize} 
Let $\big(\cS',y(\cdot)\big)$ be the set of sequences in which element $j\in [n]$ appears $Z_j$ times. 
Let $\bx'$ be chosen uniformly at random from $\big(\cS',y(\cdot)\big)$. Then, conditioned on the length of $\bx'$, be equal to $2m$, $\bx$ has the same distribution with $\bx'$.
\end{lemma}
\begin{proof}
Denote by $len(\bx')$ the length of $\bx'$. 
Since conditions (i),(ii) are satisfied by both sequences it is enough to check the case
where  $L^i(j)=0$ for $0\leq j \leq i$ and $i\in \{0,1,...,k\}$.

Fix $n,2m \in \mathbb{N}$ and $\bx''\in[n]^{2m}$. 
For $\bar{\bx}\in \{\bx,\bx',\bx''\}$ let $\cD(\bar{\bx})$ be the degree sequence of $\bar{\bx}$. Assume that $d_{\bx''}(v)\geq y_1(v)+1$ for $v\in[n]$ and $\bx''$ has length $2m$ (otherwise $\Pr(\bx=\bx'')=\Pr(\bx'=\bx''|len(\bx')=2m)=0$).

Observe that if we let  
$p_{\cD(\bx')}=\Pr(\bx'=\bx''|\cD(\bx')=\cD(\bx''), len(\bx')=2m)$ and
$p_{\cD(\bx)}=\Pr(\bx=\bx''|\cD(\bx)=\cD(\bx''))$
then,
$$p_{\cD(\bx)}=p_{\cD(\bx')}.$$

Let $\cS_d$ be a maximal subset of $\big(\cS_{n,2m}^{L^0,...,L^k},y(\cdot)\big)$ 
with the property that all elements of $\cS$ have distinct degree sequences (hence every sequence in $\big(\cS_{n,2m}^{L^0,...,L^k},y(\cdot)\big)$ can be obtained via a permutation from a unique sequence in $\cS_d$). Then for $\bx \in \cS_d$ there are exactly $\frac{(2m)!}{\underset{v\in [n]}{\prod}d_{\bx}(v)!}$ elements in $\big(\cS_{n,2m}^{L^0,...,L^k},y(\cdot)\big)$ with the same degree sequence as $\bx$. In addition, conditioned on $len(x')=2m$ 
the set of possible degree sequences for $\bx'$ is exactly the same as the set of degree sequences of elements of $\cS_d$. 
Therefore,
\begin{align*} 
\Pr(\bx=\bx'')&= p_{\cD(\bx)}\cdot  \Pr(\cD(\bx)=\cD(\bx''))
\\&= p_{\cD(\bx)}\cdot  \left( \frac{(2m)!}{\underset{v\in [n]}{\prod}d_{\bx''}(v)!}\right)
\bigg/
\left( \sum_{\bx \in \big(\cS_{n,2m}^{L^0,...,L^k},y(\cdot)\big) } 1\right)
\\&=  p_{\cD(\bx)}\cdot  \left( \frac{(2m)!}{\underset{v\in [n]}{\prod}d_{\bx''}(v)!}\right)
\bigg/
\left( \sum_{\bx \in \cS_d }\frac{(2m)!}{\underset{v\in [n]}{\prod}d_{\bx}(v)!} \right).
\end{align*}
On the other hand, 
\begin{align*}
\Pr \bigg(\bx'= \bx'' \bigg|\; len(\bx') = 2m \bigg)
&= p_{\cD(\bx')} \cdot \Pr(\cD(\bx')=\cD(\bx'')|len(\bx')=2m)
\\&= p_{\cD(\bx')} \cdot \left( \underset{v\in [n]}{\prod}\frac{ 
e^{-\lambda}\lambda^{d_{\bx''}(v)} }{d_{\bx''}(v)! f_{y_2(v)}(\lambda)}  \right) \bigg/ \left( \sum_{\bx' \in \cS_d} \underset{v\in [n]}{\prod}
\frac{ e^{-\lambda}\lambda^{d_{\bx'}(v)}}{d_{\bx'}(v)! f_{y_2(v)}(\lambda)  }\right)
\\& = p_{\cD(\bx')} \left( \lambda^{2m} \underset{v\in [n]}{\prod}
\frac{1}{d_{\bx''}(v)!}  \right)  \bigg/ \left( \lambda^{2m} \sum_{\bx' \in \cS_d} \underset{v\in [n]}{\prod}
\frac{1}{d_{\bx'}(v)!} \right) 
\\&=\Pr(\bx=\bx'').
\end{align*}
\end{proof}
We now use Lemma \ref{lem:equiv} to approximate the distribution of the  degree sequence of a random element of $\big(\cS_{n,2m}^{L^0,...,L^k},y(\cdot)\big)$.
\begin{lemma}\label{lem:degDistributions}
Let $\bx$ be chosen uniformly from $\big(\cS_{n,2m}^{L^0,...,L^k},y(\cdot)\big)$. Let $m\leq 100n\log n$ and assume that \eqref{cond} is satisfied. For $0\leq i\leq k$ and $j\geq i+1$ let $ \nu^{i}_j(\bx)$ be the number of elements $v\in [n]$ such that $y_1(v)=i$ and $d(v)=j$. Then with probability $1-o(n^{-9})$,
$$\bigg|\nu^{i}_j(\bx)- L^i(i+1) \frac{ \lambda^j}{j! f_i(\lambda)} \bigg| \leq n^{1/2}\log^2 n  \text{ for } 0\leq i\leq k \text{ and } i+1 \leq j\leq \log^2 n.$$
\end{lemma}
\begin{proof}
$\lambda$ has been chosen such that $\Pr(\cP_{\geq i}=j)= \frac{\lambda^j}{j! f_i(\lambda)}$ Thus, if we let $X$ be 
\\a $Binomial\left(L^i(i+1), \frac{ \lambda^j}{j! f_i(\lambda)}\right)$ random variable, Lemma \ref{lem:equiv} implies, 
\begin{align*}
    \Pr\bigg( \bigg|\nu^{i}_j(G_\bx)- L^i(i+1) \frac{ \lambda^j}{j! f_i(\lambda)} \bigg| > n^{1/2}\log^2 n \bigg)&
= O(m^{-0.5})    
    \Pr\big( |X-\mathbb{E}(X)|> n^{1/2}\log^2 n \big)
\\& \leq  O(m^{-0.5}) \exp\set{-O(\log^4 n)} 
    =o(n^{-9.5}),
\end{align*}
where the inequality follows from the Chernoff bound (see Theorem \ref{thm:chernoff}).
Taking union bound over $i,j$ yields the desired result.
\end{proof}
It can be shown, see for example \cite{BCFF},\cite{McK} that for a chosen uniformly at random element $x$ of  $\in \big(\cS_{n,2m}^{L_0,...,L^k},y(\cdot)\big)$ that if
$m=O(n)$ and $\sum_{i=0}^k\sum_{j=1}^{i}L^i(j)=o(n)$ then,
\begin{equation*}
\Pr(G_\bx\text{ is simple})=\Omega(1).
\end{equation*}
Hence, when $L^0,...,L^{k-1}$,  $L^k(0),L^k(1),...,L^k(k)$ are zero and $y$ maps $[n]$ to $(k,k+1)$, choosing a random element of $\bx\in \big(\cS_{n,2m}^{L_0,...,L^k},y(\cdot)\big)$ and then generating $G_\bx$ is a good model for generating
$G_{n,m}^{\d \geq k+1}$ and for any function $f(\cdot)$ such that $f(n)\to 0$ as $n\to \infty$ any property that hold with probability $1-o(f(n))$ for $G_\bx$ also hold with probability $1-o(f(n))$ for $G_{n,m}^{\d \geq k+1}$.

\subsection{Properties of $G_{n,m}^{\d\geq k+1}$}
Let $\epsilon>0$ and  $(1+\epsilon)kn/2 \leq m =O(n)$. $\lambda$ is defined by \eqref{eq:lambda}. Let $\cE$ be an occupancy event in $G_{n,m}^{\d\geq k+1}$. Denote by $G_{n,m}^{\d\geq k+1,seq}$ the random graph that is generated from the random sequence model (i.e. from first choosing a random element of $\cS_{n,2m}^{\d\geq k+1}$ and then generating the corresponding graph) and $G_{n,m}^{\d\geq k+1,Po(\l)}$ the random graph that is generated by first generating $n$ independent $ \cP_{k+1}(\l)$ random variables $P_1,P_2,...,P_n$, then choosing a random sequence in $[n]^{\sum_{i\in [n]} P_i}$ with degree sequence $P_1,P_2,...,P_n$ and finally generating he corresponding graph if $\sum_{i\in [n]} P_i$ is even. Then,
\begin{align}\label{eq:models}
    \Pr\big(G_{n,m}^{\d\geq k+1} \in \cE\big) &\leq O(1) \Pr\big(G_{n,m}^{\d\geq k+1,seq} \in \cE\big)= \Pr\bigg(G_{n,m}^{\d\geq k+1,Po(\l)} \in \cE\bigg| \sum_{i\in [n]}{P_i}=2m\bigg) \nonumber
 \\&\leq O(n^{0.5}) \Pr\big(G_{n,m}^{\d\geq k+1,Po(\l)} \in \cE\big).
\end{align} 

In the Lemmas that follow we let  $G\sim G_{n,m}^{\delta \geq k+1,seq}$.

\begin{lemma}\label{lem:degrees}
$$\Pr\bigg(\Delta(G)> \frac{10\log n}{\log\log n}\bigg) =o( n^{-9}).$$
\end{lemma}
\begin{proof}
\eqref{eq:models} implies,
\begin{align*}
 \Pr\bigg(\Delta(G)> \frac{10\log n}{\log\log n}\bigg) &
 \leq O(n^{0.5}) \cdot n  \Pr\bigg(\cP_k(\l) > \frac{10\log n}{\log\log n}\bigg)
\\& \leq O(n^{1.5}) \frac{\lambda^{\frac{10\log n}{\log\log n}}}{f_{k+1}(\lambda)(\frac{10\log n}{\log\log n})!} \leq O(n^{1.5})
\bfrac{e\lambda \log\log n}{10\log n}^{\frac{10\log n}{\log\log n}} =o(n^{-9}). \end{align*}
\end{proof}
For the next Lemma we let $\beta\in (0,1)$ be such that $\eqref{eq:beta}$ is satisfied.
\begin{lemma}\label{lem:density}
With probability $1-O(n^{-0.5})$,
\begin{itemize}
    \item[(i)] there does not exists a set $S\subset V(G)$ of size $|S|\leq k^{400}$ that spans $|S|+1$ edges in $G$.
    \item[(ii)]  there does not exists a set $S\subset V(G)$ of size $k^{200}\leq |S| \leq \beta n$ that spans $(1+10^{-10})|S|$ edges in $G$.
\end{itemize}
\end{lemma}

\begin{proof}

{\em{(i):}} For $\ell\geq 0$ let $\Phi(2\ell)$ be the number of ways to partition  a $2\ell$ element set into pairs. Then $\Phi(2\ell)=\frac{(2\ell)!}{\ell!2^\ell}$ and
\begin{align*}
\frac{\Phi(2m-2\ell)\Phi(2\ell)}{\Phi(2\ell)}
=\frac{(2\ell)!}{\ell!} \frac{m!(2m-2\ell)!}{(m-\ell)!(2m)!}
\leq (2\ell)^\ell \bfrac{1}{2(2m-2\ell)}^\ell = \bfrac{\ell}{2\ell-2\ell}^\ell.
\end{align*}
For fixed $s \leq 10^{100}$ let $r=r(s)=s+1$. Let $C=k^{400}$.  
\eqref{eq:models} implies,
\begin{align}
\Pr&( \exists S\subset V(G):  |S| \leq C  \text{ and } S \text{ spans  $|S|+1$ edges in $G$} )\nonumber
    \\&\leq O(n^{0.5}) \sum_{s= 4}^{C}  \binom{n}{s}  \sum_{\substack{d_1,d_2,...,d_s\geq k+1\\ z_1\leq d_1,...,z_s\leq d_s     \\ z_1+...+z_s=2r}} 
    \prod_{i=1}^s \frac{\lambda^{d_i}}{ d_i! f_{k+1}(\lambda)} \binom{d_i}{z_i} 
    \frac{\Phi(2m-2r)\Phi(2r)}{\Phi(2m)} \label{exp1}
    \\&\leq O(n^{0.5})  \sum_{s= 4}^{C} \binom{n}{s}   \frac{\lambda^{2r}}{f_{k+1}^s(\lambda)}  \bfrac{r}{2m-2r}^r \sum_{\substack{d_1,d_2,...,d_s\geq k+1\\ z_1\leq d_1,...,z_s\leq d_s\\ z_1+...+z_s=2r}} 
    \prod_{i=1}^s \frac{\lambda^{d_i-z_i}}{(d_i-z_i)!} 
    \nonumber
    \\&\leq  O(n^{0.5})  \sum_{s=4}^C \binom{n}{s}  \frac{\lambda^{2r}}{f_{k+1}^s(\lambda)}  
     \bfrac{r}{2m-2r}^r \sum_{D\geq 2r} \sum_{\substack{z_1,z_2,...,z_s\geq 0\\z_1+...+z_s=2r}}
      \frac{\lambda^{D-2r}s^{D-2r}}{(D-2r)!} \label{exp2}
     \\&= O(n^{0.5})  \sum_{s= 4}^C \binom{n}{s}  \frac{\lambda^{2r}}{f_{k+1}^s(\lambda)}  \bfrac{r}{2m-2r}^r \sum_{D\geq 2r} \binom{2r+s-1}{s-1}  \frac{(\lambda s)^{D-2r}}{(D-2r)!} \nonumber 
    \\&\leq   O(n^{0.5})   \sum_{s= 4}^{C}  \bfrac{en}{s}^s \frac{\lambda^{2r}}{f_{k+1}^s(\lambda)} \bfrac{r}{2m-2r}^r
    \bfrac{e(2r+s)}{s}^se^{s \lambda } \nonumber
    \\&\leq   O(n^{0.5})   \sum_{s= 4}^{C}  \bfrac{en}{s}^s \frac{\lambda^{2r}}{f_{k+1}^s(\lambda)} \bfrac{r}{2m-2r}^r
    10^se^{s \lambda } \nonumber
    \\& \leq   O(n^{0.5}) \sum_{s=4}^{C} 
    \bfrac{10e^{1+\lambda} \lambda^{2} rn }{ f_{k+1}(\lambda)s(2m-2r )}^s\bfrac{4\lambda^2r}{2m-2r }^{r -s} \label{rep}
    \\& \leq   O(n^{0.5}) \sum_{s=4}^{C} 
      \bfrac{4\lambda^2r}{2m-2r }^{r -s} =O(n^{-0.5}). \nonumber
\end{align}

{\textbf{Explanation of \eqref{exp1}}} We first choose $s$ vertices $v_1,v_2,...,v_s$ in $\binom{n}{s}$ ways. Those vertices will span a subgraph $S$ with $r$ edges. 
The degree of $v_i$ in $G$ will be $d_i$, this occurs with probability $\prod_{i=1}^s\frac{\lambda^{d_i}}{d_i! f_k(\lambda)}$, and its degree in $S$ will be $z_i$. Then, for each vertex $v_i$ we choose a set of $z_i$ out of the $d_i$ copies of $v_i$. The last term is the probability that those copies induce $\sum_{i=1}^sz_i/2$ edges when we pass form the sequence in $[n]^{\sum_{i\in [n]}d(i) }$ to the corresponding graph.

To derive \eqref{exp2} we used the following identity. For fixed $z_1,z_2,...,z_s$ if $\sum_{i=1}^s z_i=2r$ and $\sum_{i=1}^s d_i-z_i=D-2r$ then  $\sum_{\substack{z_1\leq d_1,...,z_s\leq d_s\\ d_1+...+d_s=D \\ z_1+z_2+...+z_s=2r}}\frac{(D-2r)!}{\prod_{i=1}^s(d_i-z_i)!}= s^{D-2r}$.

Now for  $k^{200}\leq s \leq \beta n$ let $r=r(s)=(1+10^{-10})s$.
Similarly to \eqref{rep} we have
\begin{align}
\Pr&( \exists S\subset V(G):  k^{200}\leq |S| \leq \beta n \text{ and } S \text{ spans $(1+10^{-10})|S|$ edges in $G$} ) \nonumber
    \\&  \leq  O(n^{0.5}) \sum_{s=  k^{200}}^{\beta n} 
    \bfrac{20e^{1+\lambda} \lambda^{2} rn }{ f_{k+1}(\lambda)s(2m-2r )}^s\bfrac{4\lambda^2r}{2m-2r }^{r -s}  =o(n^{-3}). \label{eq:beta}
\end{align}
\end{proof}

\section{The \kG Algorithm}\label{section:alg}
In this section we will describe and begin to analyze \kG. Its performance when applied to $G\sim G_{n,m}^{\delta \geq k+1,seq}$ is given by the following Theorem.
\begin{theorem}\label{thm:kGreedy}
Let $k\geq 2$,  $(k+1)/2< c=O(1)$  and $G\sim G_{n,cn}^{\d \geq k+1,seq}$. Then, with probability $1-o(n^{-9})$, \kG applied to $G$ outputs a $k$-matching $M$ of size at least $kn/2-n^{0.401}$ in $O(n)$ time.
\end{theorem}
\subsection{\kG}
\kG will be applied to the (multi)-graph $G\sim G_{n,cn}^{\d \geq k+1,seq}$. As the algorithm progresses, it makes changes to $G$ and generates a graph sequence $G_0=G \supset G_1 \supset ...\supset G_{\tau}=\emptyset$. In parallel, it grows a $k$-matching $M$. We let $M_0=\emptyset, M_1,M_2,..., M_{\tau}=M$ be the sequence of matchings that are generated.  For $v\in [n]$ we let the label of $v$ at time $t$, $l_t(v)\in \{0,...,k\}$, to be equal to $k$ minus the number of edges incident to $v$ in $M_t$ i.e. the number of edges incident to $v$ that we would like to add to $M_t$. For $0\leq t\leq \tau$ we also define, 
\begin{itemize}
\item $V=V(G_0)=\{v_1,v_2,...,v_n\}$,
\item $m_t:=|E(G_t)|$,
\item $d_t(v):=d_{G_t}(v)$, for $v\in V$,
\item $d_{M_t}(v):=|\set{ e \in M_t:v \in e} |$, for $v\in V$,
\item $Y_{\ell,j}^t:=\set{v \in [n]: l_t(v)=\ell,\,d_{t}(v)=j}$, $0\leq \ell \leq  k, 0\leq j$, i.e. the set of vertices of degree $j$ that are incident to $k-\ell$ edges in $M_t$,
\item $Y_{\ell}^t:=\cup_{j\geq \ell+1} Y_{\ell,j}^t$, for $0\leq \ell \leq k$,
\item $Z_t=\underset{{\ell\in [k], 0\leq j\leq  \ell}}{\bigcup} Y_{j,\ell}^t$, the set of ``dangerous" vertices,
\item $\zeta_t = \sum_{\ell=1}^{k} \sum_{j=1}^{\ell} j|Y_{\ell,j}^t| $,
\item $index(t)=max_{v\in V} l_t(v)$,
\item $DF_t =\set{v\in V: index(t)=l_t(v)}$.
\end{itemize}
\noindent

$DF_t$ is the set of {\emph{$M_t$-deficient vertices}} at time $t$, that is the set of vertices whose $M_t$ degree $d_{M_t}(\cdot )$ is minimized.  

Observe that if $v\in G_t$ satisfies $v\in Y^t_{\ell,j} \subseteq Z_t$ for some $j\leq \ell$ then there are $k-\ell$ edges incident to $v$ in $M_t$ and $j\leq \ell $ edges incident  to $v$ in $G_t$. Thus in the final matching $M$, $v$ can be incident to at most $k-\ell+j\leq k$ edges. In addition, if in the future an edge incident to $v$ in $G_t$ is removed but not added to the matching then,  $v$ will be incident to at most $k-1$ edges in $M$. Therefore $Z_t$ consists of {\emph dangerous vertices} whose edges we would like to add to the matching as soon as possible in order to avoid ``accidentally" deleting them. $\zeta_t$ bounds the number of those edges. \kG tries to match the vertices in $Z_t$ first.  If $Z_t$ is empty then \kG matches a vertex in $DF_t$

\begin{algorithm}[H]
\caption{{\kG}}
\begin{algorithmic}[1]
\\ Input: $G_0=G$.
\\ $t = 0$.
\While{$E(G_t)\neq \emptyset$}
{
\If{$Z_t \neq \emptyset$}
\\\hspace{10mm} Choose a  vertex $v_t \in Z_t$  proportional to its degree.
\\\hspace{10mm} Choose a random neighbor of $v_t$ in $G_t$, $w_t$. 
\Else
\\\hspace{10mm} Choose a  vertex  $v_t\in DF_t$ proportional  to its degree. 
\\\hspace{10mm} Choose a random neighbor of $v_t$ in $G_t$, $w_t$. 
\EndIf 
\\\hspace{5mm} Update: $M_{t+1} \gets M_t \cup\{\{v_t,w_t\}\}$,
\\\hspace{5mm} $l_{t+1}(v_t)\gets l_t(v_t)-1$, $l_{t+1}(w_t) \gets l_t(w_t)-1$.
\\\hspace{5mm} $l_{t+1}(u)\gets l_t(u)$ for $u\in V\setminus \{v_t,w_t\}$
\\\hspace{5mm} Delete $\{v_t,w_t\}$ from $G_t$.
\\\hspace{5mm} Delete all the edges incident to any vertex  $v$ in $G_t$ that satisfies
$l_{t+1}(v)=0$.
\\\hspace{5mm} Let $Y_{\ell,j}^t=\set{v \in [n]: l_t(v)=\ell,\,d_{G_t}(v)=j}$, $0\leq \ell \leq  k, 0\leq j\leq \ell$.
\\\hspace{5mm} Let $Y_{\ell}^t=\set{v \in [n]: l_t(v)=\ell,\,d_{G_t}(v)\geq \ell+1}$, $0\leq \ell \leq  k$.
\\\hspace{5mm} Let $G_{t+1}$ be the resultant graph.
\\\hspace{5mm} $t = t+1$. }
\EndWhile
\\ $\tau = t$
\end{algorithmic}
Remove any loops and multiple edges from $M$.
\end{algorithm}

\begin{notation}
We  let $e_t=\{v_t, w_t\}$ and $R_t$ be the set of edges deleted at step $t$. We also denote by $\cH_t$ the actions taken by \kG during the first $t$ iteration of the while loop at line 3.
\end{notation}

For $0\leq t <\tau$ and $0\leq \ell \leq k$ we let the list $L_\ell^t$ of size $\ell+2$  be defined by
 $$L^t_{\ell}(j)
   = 
  \begin{cases} 
  |Y^t_{\ell,j}| & \text{ for } 0 \leq j\leq \ell, \\ 
  |Y_\ell^t| & \text{ for }   j=\ell+1. \\ 
  \end{cases} 
$$

We let $\lambda^t$ be the value of $\lambda$ defined by \eqref{eq:lambda} with $m=m_t$ and lists $L^i=L_i^t$ for $0\leq i\leq k$.
\subsection{Uniformity}
We start by showing that $G_t$ exhibits a Markovian behaviour.

We let  $y_t(\cdot):[n]\mapsto \{0,1,...,k\} \times\{0,1,2,...,k+1\}$ be defined as follows. 
$$y_t(v):=\begin{cases}
(\ell,j) \text{ if } v\in Y^t_{\ell,j} \text{ with } 0\leq \ell\leq k \text{ and } 0\leq j\leq \ell 
\\
(\ell,\ell+1) \text{ if }  v\in Y^t_{\ell},  0\leq \ell\leq k.
\end{cases}
$$

\begin{lemma}\label{lem:uniformity}
For $t\geq 0$, suppose that $G_{t}$ is distributed as  $G_{x_{t}}$, where $x_{t}$ is a random member of $\bigg(\cS_{n,2m_t}^{L_0^t,L_1^t,...,L_k^t},y_t\bigg)$. Then given $L_0^t,L_1^t,...,L_k^t,e_{t},R_t,y_t$ and $m_t$, $G_{t+1}$ is distributed as  $G_{x_{t+1}}$, where $x_{t+1}$ is a random member of $\bigg(\cS_{n,2m_{t+1}}^{L_0^{t+1},L_1^{t+1},...,L_k^{t+1}},y_{t+1}\bigg)$.
\end{lemma} 
\begin{proof}
For $i=t,t+1$, if the graph $G_{i}$ is distributed as  $G_{x_{i}}$, where $x_{i}$ is a random member of $\bigg(\cS_{n,2m_i}^{L_0^i,L_1^i,...,L_k^i},y_i\bigg)$ then its distribution is uniform among the the ones that can be generated from $\bigg(\cS_{n,2m_i}^{L_0^i,L_1^i,...,L_k^i},y_i\bigg)$.
The statement of the lemma follows from the fact that the pair $e_t,R_t$ defines a bijection between elements of $\bigg(\cS_{n,2m_{t+1}}^{L_0^{t+1},L_1^{t+1},...,L_k^{t+1}},y_{t+1}\bigg)$ and elements of 
$\bigg(\cS_{n,2m_t}^{L_0^t,L_1^t,...,L_k^t},y_t\bigg)$ that may result to an element of $\bigg(\cS_{n,2m_{t+1}}^{L_0^{t+1},L_1^{t+1},...,L_k^{t+1}},y_{t+1}\bigg)$ through the removal of $R_t$ and the addition of $e_t$ to $M_t$.
\end{proof}

\subsection{A first analysis}
Let $E_{mult}$ be the set of loops and multiple edges deleted from $M$ at the end of \kG.
Once \kG terminates the sets $Y_{\ell,j}^{\tau}$ for $0\leq \ell \leq k$, $1\leq j$ are empty while the set $Y_{\ell,0}^{\tau}$ consists of the vertices that are incident to exactly $k-\ell$ edges in $M\cup E_{mult}$, for $0\leq \ell\leq k$. Thus $|M|+|E_{mult}|=(nk-\sum_{\ell=1}^k(k-\ell)|Y_{\ell,0}^{\tau}|)/2$. A vertex $z$ belongs to some $Y_{\ell,0}^{\tau}$ only if it ``moves" $k-\ell$ times from some set  $Y_{\ell',j}^t$, $j\leq \ell'$ to $Y_{\ell',j-1}^t$ (i.e. at time $t$ it belongs to $Y_{\ell',j}^t$ but at time $t+1$ it belongs to $Y_{\ell',j-1}^{t+1}$) 
or if it is connected to $w_t$ via a multi-edge in $G_t$. 

Let $S_t=\bigcup_{\ell'\in [k]}\bigcup_{1\leq j\leq \ell'} \left( Y_{\ell',j}^t \cap Y_{\ell',j-1}^{t+1} \right)$ and $s_t=|S_t|$. Also let $mult_t$ be the number of multi-edges in $G_t$ incident to $w_t$ counted with multiplicity and $h_t=\mathbb{I}(v_t=w_t)$ i.e. the indicator function of the event $\{e_t$ is a loop$\}$. 
Then,
\begin{equation}\label{eq:matching}
M\geq \frac{nk-\sum_{\ell=1}^k(k-\ell)|Y_{\ell,0}^{\tau}|-\sum_{0\leq i < \tau} mult_t -\sum_{0\leq i < \tau} 2h_t}{2}  = \frac{kn-\sum_{i=0}^{\tau-1}(s_t+ mult_t+2h_t)}{2}.
\end{equation}

If $index(t) \geq 3$ and $v_t=w_t$ or $v_t\neq w_t$ and $l_t(w_t)\geq 2$ then $R_t=\{e_t\}$ and $S_t=\emptyset$. Otherwise, if $index(t) \geq 3$, $v_t\neq w_t$ and $l_t(w_t)=1$ then $l_{t+1}(w_t)$ is set to zero and the edges incident to $w_t$ are deleted. In such a case $S_t \subseteq Z_t\cap \big(N_{G_t}(w_t) \setminus \{v_t\}\big)$.Therefore if $index(t) \geq 3$ we have,
\begin{align}\label{eq:m1}
\mathbb{E}(s_t+mult_t+2h_t)
&\leq \Pr(l_t(w_t)=1)\mathbb{E}\bigg(\Big|Z_t\cap\Big( N(w_t)\setminus\{v_t\}\Big)\Big|\bigg|l_t(w_t)=1\bigg) +\mathbb{E}( mult_t)+\mathbb{E}( 2h_t) \nonumber
\\&\leq \frac{\zeta_t\log n}{m_t}+\frac{\log^4 n}{m_t}.
\end{align}
At the last inequality we used Lemma \ref{lem:degrees}.

Let 
$$\tau'=\min\bigg\{t: m_i\leq n^{0.4+10^{-5}} \text{ or } \zeta_t>\log^6 n \text{ or } Y_{\ell}^t=\emptyset \text{ for }\ell\geq 3 \bigg\}.$$
To prove Theorem \ref{thm:kGreedy} it suffices to prove the following Lemma. Its proof  is given in Subsection \ref{subsection:change}.
\begin{lemma}\label{lem:changeZ_t}
With probability $1-o(n^{-9})$, for $t < \tau'$
\begin{equation}\label{eq:changeZ_t}
\text{  if $\z_t>0$ then }    \mathbb{E}(\z_{t+1}-\z_t|\cH_t)\leq -10^{-5}.
\end{equation}
\end{lemma}

{\textbf{Proof of Theorem \ref{thm:kGreedy}:}}
It follows from  \cite{AF} that $\sum_{i=\tau}^{\tau'}s_t=O( n^{0.4+ 10^{-4}})$ with probability $1-o(n^{-9})$ (In \cite{AF} a version of \kG is analyzed in which Lines 5 and 8 of \kG are substituted by ``Choose a random vertex $v_t\in Z_t$" and 
``Choose a random vertex $v_t\in Y_2$" respectively. This change does not affect the analysis in \cite{AF} which aims to show that if $\z_t>0$ 
then, the  expected change in $\zeta_t$ is not positive -it can be zero-. In addition the statements in \cite{AF} are proven to hold w.h.p. instead with probability $1-o(n^{-9})$. One can boost the probability in \cite{AF} to $1-o(n^{-9})$ by substituting  Lemma 3.3 of \cite{AF} with Lemma \ref{lem:degDistributions}.)

Lemma \ref{lem:degrees} implies that with probability $1-o(n^{-9})$,
$$|\z_{t+1}-\z_t|\leq k\Delta(G_t)\leq k\Delta(G)\leq \log n.$$ 
Therefore,
Azuma-Hoeffding inequality (see Theorem \ref{thm:AH}) gives, 
\begin{align*}
    \Pr(\exists t: &0\leq t\leq \tau' \leq m \text{ such that } \z_t\geq \log^6n)
\\    &\leq \sum_{0\leq t\leq m} \sum_{0\leq t'\leq t}\Pr(\z_t\geq \log^6n, \z_{t''}\geq 1 \text{ for } t'<t''\leq t \text{ and } 0<\z_{t'}\leq \log n)+o(n^{-9})
\\    &\leq m\sum_{0\leq t\leq m}
    \exp\bigg\{-\frac{(-\log n + \log^6n+10^{-5}t)^2}{2t\log^2 n} \bigg\}+o(n^{-9})=o(n^{-9}). 
\end{align*}
Thus \eqref{eq:matching}  \& \eqref{eq:m1} imply,
\begin{align*}
    \Pr(M<kn/2-n^{0.401})&\leq \Pr\bigg(\sum_{i=1}^{\tau'} s_t+mult_t+2h_t\geq 0.9 n^{0.401}\bigg)+o(n^{-9})
\\& \leq\sum_{0\leq t\leq m}    \exp\bigg\{-\frac{(0.9n^{0.401}-\sum_{m_i=n^{0.4+10^{-5}}}^{m}\frac{\log^7 n+\log^4 n}{m_i} )^2}{2t\log^2 n} \bigg\}+o(n^{-9})
\\&=o(n^{-9}).
\end{align*}
At the application of Azuma-Hoeffding Inequality above we used once more Lemma \ref{lem:degrees} to bound $|(s_{t+1}+mult_{t+1}+2l_{t+1})-(s_t+mult_t+2h_t)|$ by $\Delta(G) \leq \log n$. 
\qed

A vertex $u$ enters $Z_{t+1}$ at time $t\leq \tau'$ (i.e $d_{t+1}(u)>0$ and $v\in Z_{t+1}\setminus Z_t$)  only if $l_t(w_t)=1$,  $w_t\in Y^t_{1}$ and either $u\in  \bigg( \bigcup_{\ell\in[k]} Y_{\ell,\ell+1}^t\bigg)\cap\big( N_t(w_t) \setminus\{v_t\}\big)$ or $u$ is connected to $w_t$ via a multiple edge.
Hence, for $t<\tau'$ if $\zeta_t\leq \log^6 n$ then, 
\begin{align}
    \mathbb{E}(\z_{t+1}-\z_t|\cH_t)
&\leq -\mathbb{I}(\z_t>0) \nonumber
\\&+\Pr(l_t(w_t)=1)\mathbb{E} \left(  \sum_{\ell=1}^k\ell \Big|  Y_{\ell,\ell+1}^t \cap \Big( N(w_t) \setminus\{v_t\} \Big)\Big|\bigg|l_t(w_t)=1\right)+o(1). \label{eq:change_zt}
\end{align}

We will control $\z_t$ by showing that $\Pr(l_t(w_t)=1)$ stays sufficiently small throughout the execution of \kG. 

\section{Notation-Preliminaries}
\begin{notation}
We say that a sequence of events $\{\mathcal{E}_n\}_{n\geq 1}$
holds {\em{with sufficiently high probability}} (w.s.h.p.\@ in short) if $\lim_{n \to \infty}\Pr(\mathcal{E}_n)=1-o(n^{-9})$. 
\end{notation}

For a function $f(\cdot):[\tau]\mapsto \mathbb{R}$ we let $D_t(f)=f(t+1)-f(t)$.

For $\lambda \geq 0$ and $0\leq \ell \leq k$ we let $Po_{\geq \ell}(\lambda)$ be the truncated at $\ell$ $Poisson(\lambda)$ random variable. Thus
$$\Pr(Po_{\geq \ell}(\lambda)=r)=\frac{e^{-\lambda}\lambda^r}{r!\sum_{i\geq \ell}\frac{e^{-\lambda}\lambda^i}{i!}}=\frac{\lambda^r}{r!f_\ell(\lambda)} \hspace{20mm} \text{ for }\ell, \ell+1, \cdots .$$
We denote the expected value of $Po_{\geq \ell}(x)$ by $\lambda_\ell(x)$. We may write $\lambda_\ell$ in place of $\lambda_\ell(x)$ when the value of $x$ is clear from the context. 
We let 
$$q_{\ell,r}=q_{\ell,r}(\lambda)=\begin{cases} \frac{r\Pr(Po_{\geq \ell}(\lambda)=r)}{\mathbb{E}(Po_{\geq \ell}(\lambda))}, &r= \ell, \ell+1, \cdots ,
\\0& \text{otherwise}. \end{cases} $$ 
and
$$q_{\ell,r}^t=q_{\ell,r}(\lambda^t).$$
At time $t$ Lemmas \ref{lem:degDistributions} \& \ref{lem:degrees} imply that $\sum_{v\in {Y_{\ell}^t}} d_t(v)=
|Y_\ell^t|\mathbb{E}(Po_{\geq \ell+1}(\lambda^t))+o(m_t)$ and
that $\sum_{v\in {Y_{\ell,r}^t}} d_t(v)=r|Y_\ell^t|\Pr(Po_{\geq \ell+1}(\lambda^t)=r)+o(m_t)$,
for $\ell\in[k]$ and $r>\ell$.
Thus at time $t$, if $Y_{\ell}^t=\Omega(m_t)$ then, $q_{\ell,r}^t$ equals up to an $o(1)$ additive factor to  the probability a vertex chosen from $Y_\ell^t$ proportional to its degree has degree $r$.

For $t\geq 0$, $\ell\in [k+1]$, and $j>0$ we let 
$$p^t_{\ell,j} = \frac{j Y_{\ell,j}^t}{2m_t}$$
be the probability that a vertex chosen at proportional to its degree belongs to $Y^t_{\ell,j}$.
We let
$$p^t_\ell =\sum_{j\geq \ell+1} \frac{j Y_{\ell,j}^t}{2m_t},$$
be the probability that a vertex chosen at proportional to its degree belongs to $Y^t_{\ell}$. We also let $p^t\in \mathbb{R}^{k+1}$ be defined by $(p^t)_\ell=p_\ell^t$ for $1\leq \ell \leq k+1$.
Observe that while $m\geq n^{0.4+10^{-5}}$ and $\zeta_t\leq \log^6 n$ we have that
$$\sum_{i=1}^{k}p_k^t=1-o(1).$$
For $2\leq \ell \leq k$ we define the stopping time $\tau_\ell$ by,
$$\tau_\ell=\min\{t: Y_i^t=\emptyset \text{ for } i\geq \ell \text{ or } m_t\leq n^{0.4+10^{-5}} \text{ or } \zeta_t \geq \log^6 n \}.$$
Recall, $\tau'=\min\{t: m_t\leq n^{0.4+10^{-5}} \text{ or } Y_\ell^t=\emptyset \text{ for } 3\leq \ell \leq k+1\}$ and hence $\tau_3\leq \tau'$. 
Finally we let $\cH_t$ be the $t$-tuple $(h_0,h_1,...,h_{t-1})$ where $h_i=(e_i,R_i,Z^i)$. Thus $\cH_t$ encodes the actions taken by the \kG in  the first $t$ steps.
\begin{lemma}\label{lem:lk}
Let $\lambda\geq 0$ and $\ell\geq 1$. At the $t^{th}$ iteration of \kG we have
 $$\mathbb{E}\big[(|N(w_t)|-1) \mathbb{I}( w_t \in Y_\ell^t) \big]= p_\ell^t \lambda_{\ell}^t+o(1).$$
\end{lemma}
\begin{proof}
$ p_\ell^t= \frac{\sum_{v\in Y_\ell^t} d_t(v)}{2m_t}$. Therefore, with $\lambda=\lambda^t$
\begin{align*}
    \mathbb{E}\big[(|N(w_t)|-1) \mathbb{I}( w_t \in Y_\ell^t) \big]&
    =  \sum_{v\in Y_\ell^t} (d_t(v)-1) \frac{ d_t(v) }{ 2m_t}=\sum_{j\geq \ell+1} (j-1) \frac{ j|Y_\ell^t| \frac{  \lambda^j /j!}{ f_{\ell+1}(\lambda )}  }{2m_t}+o(1) 
\\  &=  \frac{|Y_\ell^t| \sum_{i \geq \ell+1} \frac{ i  \lambda^i /i!}{f_{\ell+1}(\lambda )}   }{2m_t}\frac{ \sum_{j\geq \ell+1} (j-1) j\frac{  \lambda^j/j!}{f_{\ell+1}(\lambda)}  }{\sum_{i \geq \ell+1} \frac{ i  \lambda^i /i!}{ f_{\ell+1}(\lambda )}}
+o(1) 
\\&
=  p_\ell^t  \frac{ \sum_{j\geq \ell} j\frac{\lambda^j }{j!}}{\sum_{i \geq \ell}\frac{\lambda^i }{i!}}   +O(1) = p_\ell^t  \lambda_{\ell}^t +o(1).
\end{align*}
The $o(1)$ error terms are due the applications of 
Lemmas \ref{lem:degDistributions} \& \ref{lem:degrees} at the second and fourth equalities.
\end{proof}

\begin{lemma}\label{lem:lambdamonotone}
Let $r\geq 1$ and $x\geq 0$. Then,
$$\lambda_{r-1}(x)\leq \lambda_r(x).$$
\end{lemma}
\begin{proof}
\begin{align*}
    \lambda_{r-1}(x)
&=\frac{\sum_{i\geq r-1}ix^i/i! f_{r-1}(x) }{\sum_{i\geq r-1} x^i/i! f_{r-1}(x)}
=\frac{x\sum_{i\geq r-2}x^i/i! }{\sum_{i\geq r-1} x^i/i!}
=x+\frac{x^{r-1}/(r-2)! }{\sum_{i\geq r-1}x^i/i!}
\\& =x+\frac{1 }{\sum_{i\geq r-1} x^{i-r+1}(r-2)!/i!} \leq x+\frac{1 }{\sum_{i\geq r} x^{i-r}(r-1)!/i!}=\lambda_r. 
\end{align*}
\end{proof}

\begin{lemma}\label{lem:lambdabound}
Let $r\geq 1$ and $x\geq 0$. Then,
$$\max\set{x,r}\leq \lambda_{r}(x) \leq x+r.$$
Hence,
$$\lambda_{r}(x)\leq \lambda_{r-1}(x)+r.$$
\end{lemma}
\begin{proof}
Indeed,
\begin{align*}
    \max\set{x,r}\leq \lambda_{r}(x)
=x+\frac{x^{r}/(r-1)! }{\sum_{i\geq r}x^i/i!}=x+r\cdot \frac{x^{r}/r! }{\sum_{i\geq r}x^i/i!}
\leq x+r.
\end{align*}
\end{proof}
\subsection{Large deviation bounds}
In this subsection we collect some standard results which are used throughout the paper, (see e.g. Chapter 22 of \cite{abook}).
\begin{theorem}[Chernoff Bound]\label{thm:chernoff}
$$\Pr\bigg(Binomial(n,p)-np\geq (1+\epsilon)np\bigg) \leq \exp \set{-\frac{\epsilon^2(np)}{3}}.$$
\end{theorem}

\begin{theorem}[McDiarmid's Inequality]\label{McD}
Let $Z=Z(W_1,W_2,...,W_n)$ be a random variable that depends on $n$ independent random variables $W_1,W_2,...,W_n$. Suppose that there exists constants $c_1,c_2,...,c_n$ such that
$$|Z(W_1,W_2,...,W_i,...,W_n)-Z(W_1,W_2,...,W_i',...,W_n)|\leq c_i$$

for all $i\in[n]$ and $W_1,W_2,...,W_n,W_i'$. Then,
$$\Pr(Z-\mathbb{E}(Z)\geq t)\leq 2\exp\set{-\frac{t^2}{2\sum_{i=1}^n c_i}} .$$
\end{theorem}

\begin{theorem}[Azume-Hoeffding Inequality]\label{thm:AH}
Let $\{X_i\}_{i=0}^n$ be a martingale and assume that there exists a sequence $\{c_i\}_{i=1}^n$ such that $|X_i-X_{i-1}|\leq c_i$ for $i \in [n]$. In addition assume that $X_0$ is constant. Then, for $t>0$,
$$\Pr(|X_n-X_0|\geq  t) \leq 2\exp\set{-\frac{t^2}{2\sum_{i=1}^n c_i}}.$$
\end{theorem}

\section{Analysis of \kG}\label{sec:analysis}
The presented analysis of \kG  aims to prove Lemma \ref{lem:changeZ_t} and is driven by \eqref{eq:change_zt}. Given the notation introduced in the previous section \eqref{eq:change_zt}  can be rewritten as,
\begin{align}\label{eq:echange_zt}
    \mathbb{E}(\z_{t+1}-\z_t|\cH_t)&=-\mathbb{I}(\z_t>0)+p_1^t \lambda_1^t \sum_{\ell=1}^k \ell p_\ell^tq_{\ell+1,\ell+1}^t+o(1).
\end{align}
Now let $$t^*=\bfrac{1}{40ck}^{4} n.$$

To bound the RHS of \eqref{eq:echange_zt} we will show that for $d\in\{k+1,k,...,7\}$ we can confine the process $p^{t^*},p^{t^*+1},...,p^{\tau_{d-1}-1}$ inside a polyhedron $B_d \subset \mathbb{R}^{k+1}$ of the form $$B_d=\{x\in \mathbb{R}^{k+1}:x_{i}\geq \alpha_ix_{i-1} \text{ for } 2\leq i\leq d-1\}$$ such that the following properties hold:
\\ \textbf{Property I:}  $p^{t^*}\in B_{k+1}$, $p^{t^*}_i\geq \alpha_ip^{t^*}_{i-1}+\epsilon$ for $i\in\{2,3,...,k+1\}$ and some $\epsilon>0$,
\\ \textbf{Property II:} there exists $\epsilon'>0$ such that for $d\in\{k+1,k,...,7\}$, $t^*\leq t\leq  \tau_{d-1}-1$ and $i\in\{2,3,...,d-1\}$, if $p^t_i-a_{i}p_{i-1}^t\leq \epsilon '$ and $p^t\in B_d$ then 
$$\mathbb{E}[(p^{t+1}_i-a_{i}p_{i-1}^{t+1})-(p^t_i-a_{i}p_{i-1}^t)|\cH_t] >0,$$ 
\\ \textbf{Property III:} for $d\in\{k+1,k,...,3\}$ and  $\tau_{d} \leq t < \tau_{d-1}$, if  $p^t\in B_d$ then \eqref{eq:changeZ_t} holds.

Property I implies that the process at time $t^*$ lies inside $B_d$ and away from its boundary. Property II will imply that if the process starts inside $B_d$ and away from its boundary then it stays inside $B_d$ until time $\tau_{d-1}$. Thus Properties I \& II will imply that $p^t \in B_d$ for $d=k+1,k,...,7$ and $t^*\leq t\leq \tau_{d-1}$ . Thereafter Property III will imply that \eqref{eq:changeZ_t} is satisfied for $\tau^*\leq t\leq \tau_6$. For $0\leq t\leq \tau^*$ we will argue that $p_1$ is tiny and this will be all that is needed to show that \eqref{eq:changeZ_t} is satisfied. For $\tau_6\leq t \leq \tau'$ we will confine the process in a more careful and calculated way. This part of the proof takes place  in the Appendix.

Henceforward to easy the presentation we  drop the time index $t$ when it is clear from the context.

\subsection{ Defining $B_d$ }
$B_d$ will be defined by a set of inequalities of the form $x_{i}\geq \alpha_{i}x_{i-1}$ for $2\leq i\leq d-1$. The constants $\alpha_i$ will be independent of $d$, are defined by \eqref{alphas} stated later and we will prove that they lie in the interval $[1.55,1.552]$. Hence if $p^t\in B_d$ then the sequence $p_1^t$, $p_2^t$, $p_3^t$, $\cdots$, $p_{d-1}^t$ increases in a geometric fashion. This will imply that $p_1^t$ will be sufficiently small so that if $\z_t>0$ then the RHS of \eqref{eq:echange_zt} is strictly negative. The equations defining $a_i$, $2\leq i\leq d-1$, originate from the  calculations related to Property II
.  

For $r\in \mathbb{N}^+$ and $\lambda>0$ we let
\begin{align}\label{grl}
g(r,\lambda):=a_r\lambda_{r+1}-(a_r+1)\lambda_{r}+\lambda_{r-1}.
\end{align}
We let $$ \alpha_2:=1.55.$$
and
$$a_2':=\underset{\lambda \geq 0}{sup}\bigg\{
a_{2}-\frac{g(2,\lambda)}{\lambda_{3}}
+\frac{0.55\lambda_1}{(1.55^{3}-1)\lambda_{3}}  \bigg[ 3 q_{3,3} - 2 q_{2,2}\bigg]
 \bigg\}.$$
For $r\geq 3$ let 
$$a_r':=\underset{\lambda \geq 0}{sup}\bigg\{
a_{r}-\frac{g(r,\lambda)}{\lambda_{r+1}}
+\frac{0.55\lambda_1}{(1.55^{r}-1)\lambda_{r+1}}  \bigg[ (r+1) q_{r+1,r+1} - rq_{r,r}\bigg]
 \bigg\}.$$
We define $a_{r+1}$ by
\begin{equation}\label{alphas}
a_{r+1}:=\begin{cases}
\max\{a_{r},a_r'\}+10^{-5} \hspace{10mm} \text{ for } 2\leq r \leq 24,
\\ \max\{a_{r},a_r'\}+2^{-r} \hspace{10mm} \text{ for } 25\leq r\leq k.
\end{cases}
\end{equation} 
For $d\in [k+1]$ we let
$$B_d:=\{ x\in \mathbb{R}^{k+1}: x_r \geq \alpha_r x_{r-1} \text{ for } 2\leq r \leq d-1\}.  $$

We now proceed to bound $a_r$ from above for $3\leq r\leq k$. As it can be seen by the equalities defining $a_r$, the whole problem boils down to an optimization problem. For small values of $r$, treated at Lemma \ref{lem:boundsalphasmall}, we resolve this problem computationally. Its proof is presented in Appendix \ref{appendix:boundsalphasmall}. The proof of Lemmas \ref{lem:boundsalphasmall} and \ref{lem:boundsalphalarge} make use of Lemma \ref{lem:lambdaconvex} stated below. Its proof is presented in Appendix \ref{appendix:lambdaconvex}.
\begin{lemma}\label{lem:lambdaconvex}
For $r \geq 2$ and $\lambda\geq 0$ we have that $g(r,\lambda)\geq 0$.
\end{lemma}
\begin{proof}
See Appendix \ref{appendix:lambdaconvex}
\end{proof}

\begin{lemma}\label{lem:boundsalphasmall}
For $2\leq r\leq 24$ we have that $1.55\leq a_{r+1}\leq 1.551$.
\end{lemma}
\begin{proof}
See Appendix \ref{appendix:boundsalphasmall}
\end{proof}

\begin{lemma}\label{lem:boundsalphalarge}
For $r\geq 25$ we have that $1.55\leq a_{r+1}\leq 1.552$.
\end{lemma}
\begin{proof}
Lemma \ref{lem:boundsalphalarge} will follows from the following claim and Lemma \ref{lem:boundsalphasmall}.

\begin{claim}\label{claim:alphas1}
For $r\geq 25$ we have that $a_r'-a_r\leq 2.5\cdot10^{-4}\cdot 1.45^{-r+25}$.
\end{claim}

Claim \ref{claim:alphas1} and Lemma \ref{lem:boundsalphasmall} imply that for $r\geq 25$,
$$1.55=a_2\leq a_{r+1}\leq a_{25}+\sum_{r\geq 25} (2.5\cdot10^{-4}\cdot 1.45^{-r+25}+2^{-r}) \leq 1.551+\frac{2.5\cdot10^{-4}}{1-1/1.45}+2^{-24}\leq 1.552.$$

{\textbf{Proof of Claim \ref{claim:alphas1}:}}
Lemma \ref{lem:lambdaconvex} and $q_{\ell,\ell}\leq 1$ imply that for $r\geq 25$
\begin{align*}
    a_r'-a_r&\leq 
    \underset{\lambda\geq 0}{\sup}\bigg\{ \frac{0.55\lambda_1}{(1.55^{r}-1)\lambda_{r+1}} (r+1)  
 \bigg\}
\leq \frac{0.55(r+1)}{1.55^{r}-1}  \\& \leq \frac{0.55\cdot 26}{1.55^{25}-1}\prod_{i=25}^{r-1}\frac{\frac{i+2}{1.55^{i+1}-1}}{\frac{i+1}{1.55^{i}-1}} 
\leq 2.5\cdot10^{-4}\cdot 1.45^{-r+25}.
\end{align*}
At the second inequality we used that $\lambda_1\leq \l_{r+1}$, implied by Lemma \ref{lem:lambdamonotone}.
\end{proof}

\subsection{Property III}
In this section we show that for $\tau_d\leq t <\tau_{d-1}$, if $p^t\in B_d$ then \eqref{eq:changeZ_t} holds. We split this statement into a statement concerning large values of $d$ and a statement concerning small values of $d$. We prove the second one in the Appendix. Before proceeding to the proof of the first one we first proof two auxiliary Lemmas.
\begin{lemma}\label{lem:p1}
 If $p_t \in B_d$ then, 
 \begin{equation}\label{eq:p2bound}
 p_1\leq \frac{0.55}{1.55^{d-1}-1} \end{equation}
\end{lemma}
\begin{proof}
$p_t \in B_d$ implies  that $p_i\geq  1.55^{i-1}p_1$ for $1\leq i\leq d-1$. Therefore,
\begin{align*}
p_1& \leq \frac{p_1}{\sum_{i=1}^d p_i } \leq  \frac{p_1}{\sum_{i=1}^{d-1} 1.55^{i-1} p_1} = \frac{0.55}{1.55^{d-1}-1}.
\end{align*}
\end{proof}

\begin{lemma}\label{lem:calcipq}
For $d\in[k+1]$ and $\tau_{d} \leq t \leq \tau_{d-1}-1$, if $p_t\in B_d$ then,
\begin{equation}\label{eq:ipq}
    p_1\lambda_1\sum_{i=1}^{d} ip_i q_{i+1,i+1}  \leq \frac{0.55d(d+1)}{1.55^{d-1}-1}.
\end{equation}
\end{lemma}
\begin{proof}
\begin{align*}
p_1\lambda_1\sum_{i=1}^{d} ip_i q_{i+1,i+1}& \leq \frac{0.55}{1.55^{d-1}-1} \cdot\sup_{\lambda\geq 0} \bigg\{ \lambda_1 \cdot    \sum_{i=1}^{d} i p_i \frac{(i+1)\frac{e^{-\lambda}\lambda^{i+1}}{(i+1)!}}{\lambda-\sum_{j=0}^i \frac{j e^{-\lambda}\lambda^{j}}{j!}}\bigg\}  
\\& \leq \frac{0.55}{1.55^{d-1}-1} \cdot \sup_{\lambda\geq 0} \bigg\{ (\lambda+1) \cdot
    \sum_{i=1}^{d} i p_i \frac{\lambda^{i}/i!}{e^{\lambda}-\sum_{j=0}^{i-1}  \lambda^{j}/j! } \bigg\} 
    \\& \leq \frac{0.55}{1.55^{d-1}-1}
    \cdot \sup_{\lambda\geq 0} \bigg\{ 
    \sum_{i=1}^{d} i p_i \frac{\lambda^{i}/i!+(i+1)\lambda^{i+1}/(i+1)!}{\sum_{j\geq i}  \lambda^{j}/j! } \bigg\}  \nonumber
        \\& \leq \frac{0.55}{1.55^{d-1}-1}
    \cdot  
    \sum_{i=1}^{d} i(i+1) p_i  \leq \frac{0.55d(d+1)}{1.55^{d-1}-1} .
\end{align*}
At the first inequality we used \eqref{eq:p2bound} and at the second one Lemma \ref{lem:lambdamonotone} to deduce that $\lambda_1(x) \leq x+1$.
\end{proof}

\begin{lemma}\label{03}
For $11\leq d \leq k+1$ and $\tau_{d} \leq t \leq \tau_{d-1}-1$ if $p^t \in B_d$ then \eqref{eq:changeZ_t} holds.
\end{lemma}
\begin{proof}
Let $\tau_{d} \leq t \leq \tau_{d-1}-1$. Equations \eqref{eq:echange_zt} and \eqref{eq:ipq} imply that if  $p^t\in B_d$ then,
\begin{align*}
    \mathbb{E}(D_t(\z)|\cH_t)&\leq -1+\frac{0.55d(d+1)}{1.55^{d-1}-1}  +o(1) \leq -0.001. 
\end{align*}
\end{proof}

\begin{lemma}\label{02}
For $3\leq d \leq 10$ and $\tau_{d} \leq t \leq \tau_{d-1}-1$ if $p^t \in B_d$ then \eqref{eq:changeZ_t} holds.
\end{lemma}
\begin{proof}
See Appendix \ref{appendix:02}
\end{proof}

\subsection{Confining the process $\{p^t\}_{\tau_d\leq t< \tau_{d-1}}$ inside $B_d$}\label{sub:confine}
In this subsection we show that $p^t \in B_d$ for $t^*\leq t \leq \tau_{d-1}-1$.  We start by showing that $p^{t^*}\in B_d$
Thereafter we show that for $t^* \leq t<\tau_{d-1}$ if $p^t\in B_d$ and $p^t$ is close to some hyperplane defining $B_d$ then in expectation $p^{t+1}$ will move away from that hyperplane. We then use a standard martingale argument to argue that the process will never come too close to such a hyperplane and therefore it will never cross it, thus staying inside $B_d$. 

Part (i) of Lemma \ref{lem:initial} implies that equation \eqref{eq:changeZ_t} is satisfied for $t\leq t^*$. Its second part implies that $p^{t^*} \in B_d$ for $3\leq d \leq k+1$ while its third part states that starting from $t=t^*$ until the process $\{p^t\}_{t=t^*}^{\tau}$ exits $B_3$ we have that $p_1^t$ is bounded below by a constant.   
Recall, $t^*=\bfrac{1}{40ck}^{4} n$.

\begin{lemma}\label{lem:initial}
Let $C_1=p_1^{t^*}$. W.s.h.p.\@,
\begin{itemize}
\item[(i)]  $p_1^t \lambda_1^t \sum_{\ell=1}^k \ell p_\ell^tq_{\ell+1,\ell+1}^t\leq 0.5$ for $t\leq t^*$,
\item[(ii)] $C_1=\Omega(1)$ and $p^{t^*}_j-a_j  p^{t^*}_{j-1}\geq C_1$ for $2\leq j\leq k$,
\item[(iii)] let $\sigma^*=\min\{\tau_3, \min\{t\geq t^*: p^t \notin B_3\}\}$ then $p_1^t\geq \min\{0.04,C_1/2\}$ for $t^*\leq t\leq \sigma^*$.
\end{itemize}
\end{lemma}
\begin{proof}
See Appendix \ref{appendix:initial}.
\end{proof}

\begin{lemma}\label{lem:lowerconstrains}
Let $3\leq d \leq k+1$ and $2\leq r\leq d-2$. Let $t$ be such that $t\leq \tau_{d-2}-1$. If $p^t \in B_d$, $p_1^t\geq C_1/2$ and 
\begin{align}\label{eq:r-}
p^t_{r}-\alpha_r p^t_{r-1} \leq n^{-0.1}
\end{align}
then,
\begin{align}\label{eq:lowerconstrains}
    E( D_t(2mp_{r}-\alpha_r 2mp_{r-1})  |\cH_t)>10^{-5}2^{-r}C_1.
\end{align}
\end{lemma}
\begin{proof}
$p^t\in B_d$ implies that,
\begin{equation}\label{eq:r+1}
p_{r+1}^t \geq \a_{r+1} p_r^t
\end{equation}
Hence,
\begin{align}
&\mathbb{E}[D_t(2mp_{r}-2\a_{r} mp_{r-1})| \cH_t ] \geq  \lambda_{r+1} p_{r+1}-(\lambda_{r}+1)p_r-a_r \lambda_{r}p_r+\a_r( \lambda_{r-1}+1)p_{r-1}\nonumber
\\ &\hspace{4mm}+ p_1 \lambda_1 \bigg[- (r+1) p_r q_{r+1,r+1} + \a_r r p_{r-1} q_{r,r}\bigg]  +o(1) \nonumber
\\ &\geq    \lambda_{r+1} a_{r+1} p_{r}
-(\lambda_{r}+1)p_r-a_r \lambda_{r}p_r+( \lambda_{r-1}+1)(p_{r}- o(1)) \label{eq0}
\\ &\hspace{4mm} + p_1 \lambda_1 \bigg[- (r+1) p_r q_{r+1,r+1} + r (p_{r}-o(1)) q_{r,r}\bigg] +o(1) \nonumber
\\&  \geq \bigg\{(a_{r+1}-a_r) \lambda_{r+1} +g(r,\lambda) 
-  p_1 \lambda_1 \bigg[ (r+1) q_{r+1,r+1} - r q_{r,r}\bigg]\bigg\}  p_r+o(1) \label{eq1}
\\&  \geq \min\{10^{-5},2^{-r}\} \lambda_{r+1} p_r 
+o(1)  \geq 10^{-5}2^{-r}(r+1)\cdot 1.5^{r-1}p_1 +o(1) \geq 10^{-5}2^{-k}C_1.\label{eq:new01}
\end{align}
At \eqref{eq0} we used \eqref{eq:r-} and \eqref{eq:r+1}. At \eqref{eq:new01} we used   the definition of $\alpha_{r+1}$  (see \eqref{alphas}) and $p^t\in B_d$. At the very last inequality  we used part(iii) of Lemma \ref{lem:initial}.

{\textbf{Explanation of \eqref{eq0}}}: Let $\ell=index(t)$. $t< \tau_{d-2}$ implies that $\ell\geq d-1$. $v_t\in Z_t \cup Y_{\ell}^t$, thus if $\ell=d-1$ matching $v$ contributes by a non-negative positive amount to the quantity in question else it does not affect it at all.
$w_t \in Y_i^t$ with probability $p_i^t$. Also by  Lemma \ref{lem:lk} we have  $\mathbb{E}(|N(w_t)| p_i^t \mathbb{I}( w_t \in Y_i^t))= \lambda_{i}+1+o(1)$. Hence for $i\in \{r-1,r,r+1\}$, $w_t\in Y_{i}$ with probability $p_{i}^t$. In such an event it is incident, in expectation, to $\lambda_{i}+1$ edges. Removing $e_t$ removes an edge incident to $w_t$ and moves $w_t$ from $Y_{i}^t$ to $Y_{i-1}^t$. This increases the number of half edges incident to $Y_{i-1}^t$ by $\lambda_{i}$ in expectation (an edge spanned by $Y_{i-1}^t$ counts as two half edges while an edge with only one endpoint in $Y_{i-1}^t$ counts as a single half edge). This explains the first line of \eqref{eq0}. For second line of equality \eqref{eq0} note that $w_t$ belongs to $Y_1^t$ with probability $p_1$. Removing $w_t$ removes an edge from the $\lambda_1 p_i^tq_{i+1,i+1}$ in expectation neighbors of $w_t$ that belong to $Y_{i,i+1}^t$ for  $i\in\{r,r+1\}$. Those vertices are now incident to $i$ edges and are transferred to $Y_{i,i}^t$, hence outside $Y_i^t$. As a result the number of half edges incident to $Y_i^t$ is decrease in expectation by $(i+1)p_i^tq_{i+1,i+1}$.
\end{proof}

A  statement similar to the one of Lemma \ref{lem:lowerconstrains} but for $r=d-1$ is given by the Lemma \ref{lem:upperconstrain}. For its proof we need the following lemma which we use to approximate the deterministic quantity $\mathbb{I}(\z_t=0)$ appearing in the calculation of $\E(D_t(2mp_{d-1}-\alpha_r 2mp_{d-2})|\cH_t).$

\begin{lemma}\label{eq:sumzetas}
W.s.h.p.\@ for $4\leq d\leq k+1$ and $ \tau_d \leq t \leq \tau_{d-1}-\log^8n$,
\begin{align}
    \sum_{i=t}^{t+\log^{8}n-1}\mathbb{I}(\z_i=0) \geq \sum_{i=t}^{t+\log^{8}n-1} \left(1-p_1^i\lambda_1^i\sum_{j=1}^d j p_j^iq_{j+1,j+1}^i -10^{-10} \right).
\end{align}
\end{lemma}
\begin{proof}
Let $F(\z_t)$ be the increase in $\z_t$ that may occur after removing $e_t$. Then, $F(\z_t)=\z_{t+1}=\z_{t+1}-\z_t$ if $\z_t=0$ and $F(\z_t)=\z_{t+1}-\z_t+1$ if $\z_t>0$. In addition,  
$$\z_{t+\log^{8}n} = \z_{t}+ \sum_{i=t}^{t+\log^{8}n -1}[F(\z_{i })-1+\mathbb{I}(\z_{i}=0)] \text{ for } \tau_d \leq t \leq \tau_{d-1}-\log^8n.$$
Hence, 
\begin{align*}
  \sum_{i=t}^{t+\log^{8}n-1} \mathbb{I}(\z_{i}=0)
 &= (\z_{t+\log^{8}n}- \z_{t})
 + \sum_{i=t}^{t+\log^{8}n-1} 1-F(\z_{i}). 
 \\& \geq - \log^6n+
 \sum_{i=t}^{t+\log^{8}n-1} 1-F(\z_{i}). 
\end{align*}
At the last inequality we used that $\zeta_t\leq \log^6 n$ for $t< \tau_{3}$. Now observe that $F(\z_t)\neq 0$ iff $w_t\in Y_1^t$. In such a case $F(\z_t)=\sum_{v\in N(w_t)\setminus \{v_t\}} \sum_{j=1}^{k+1} j \mathbb{I}( v\in Y_{j,j+1}^t).$ Therefore, 
$$F(\z_t)\leq (k+1)\D(G) \leq \log n$$
and 
$$\mathbb{E}(F(\z_t)|\cH_t)= p_1^t\lambda_1^t\sum_{j=1}^d j p_j^tq_{j+1,j+1}^t+o(1).$$
Thus,
\begin{align*}
    &\Pr\left(   \sum_{i=t}^{t+\log^{8}n-1}\mathbb{I}(\z_t=0) \leq  \sum_{i=t}^{t+\log^{8}n-1} \left(1-p_1^i\lambda_1^i\sum_{j=1}^d j p_j^iq_{j+1,j+1}^i -10^{-10}\right)\right)
    \\ &\leq  \Pr\left(     \sum_{i=t}^{t+\log^{8}n-1}1-F(\z_i) \leq \sum_{i=t}^{t+\log^{8}n-1} \left(1-\mathbb{E}(F(\z_i)|\cH_i\right))- 10^{-11}\log^8 n\right)
    \\ &= \Pr\left(      \sum_{i=t}^{t+\log^{8}n-1}\mathbb{E}(F(\z_i)|\cH_i )-F(\z_i) \leq - 10^{-11}\log^8 n\right)
    \leq \exp \left\{- \frac{(10^{-11}\log^8 n)^2}{2\log^2 n\cdot \log^8 n}\right\}=o(n^{-9}).
\end{align*}
The last inequality follows from Azuma-Hoeffding inequality (see Theorem \ref{thm:AH}).
\end{proof}

\begin{lemma}\label{lem:upperconstrain}
Let $7\leq d \leq k+1$. Let $t$ be such that $\tau_d\leq t  <  \tau_{d-1}-\log^8 n$. If $p^{t+i} \in B_d$, $p_1^{t+i}\geq C_1/2$ and 
\begin{align}\label{eq:r-2}
p^{t+i}_{d-1}-\alpha_{d-1} p^{t+i}_{d-2} \leq  n^{-0.1}
\end{align}
for $i=0,1,...,\log^8n$ then,
$$ \sum_{i=0}^{\log^{8}n-1}\E\big[D_{t+i}\big(2mp^{}_{d-1}-\alpha_{d-1}2m p^{}_{d-2} \big)  \big|\cH_{t+i}\big]\geq 10^{-5}\log^{8} n.$$
\end{lemma}
\begin{proof}
Similarly to the derivation of \eqref{eq1} we have,
\begin{align}
\E[D_t&(2mp_{d-1}-2\a_{d-1} mp_{d-2})| \cH_i ] =\lambda_{d} \mathbb{I}(\zeta=0)+\lambda_{d}p_{d} -(\lambda_{d-1}+1)p_{d-1}-a_{d-1} \lambda_{d-1}p_{d-1} \nonumber\\&\hspace{4mm}+   \a_{d-1} (\lambda_{d-2}+1)p_{d-2}    + p_1 \lambda_1 \bigg[- d p_{d-1} q_{d,d}+ \a_{d-1} (d-1) p_{d-2} q_{d-1,d-1}\bigg] +o(1)  \nonumber
\\ &\geq  \lambda_{d}  \mathbb{I}(\zeta=0)+   \lambda_{d} p_{d}-(\lambda_{d-1}+1)p_{d-1}-a_{d-1} \lambda_{d-1}p_{d-1}+(\lambda_{d-2}+1)(p_{d-1}-o(1))   \nonumber
\\ &\hspace{4mm} + p_1 p_{d-1} \lambda_1 \bigg[- d  q_{d,d}+(d-1) q_{d-1,d-1}+o(1)\bigg]  +o(1) \nonumber
\\&  \geq [\mathbb{I}(\zeta=0)+p_{d}-a_dp_{d-1}]\lambda_d
\nonumber \\& \hspace{4mm}
+  \bigg\{(a_{d}-a_{d-1}) \lambda_{d} +g(d-1,\lambda) 
 -  p_1 \lambda_1 \bigg[ d q_{d,d} - (d-1) q_{d-1,d-1}\bigg]\bigg\} p_{d-1}+o(1) .\label{eq12}
\end{align}

We split the rest of the proof of Lemma \ref{lem:upperconstrain} into the following 2 claims.

\begin{claim}\label{claim:stopsmall}
Let $7\leq d \leq 20$. Let $t$ be such that $\tau_d\leq t < \tau_{d-1}-\log^8 n$. If $p^{t+i} \in B_d$, $p_1^{t+i}\geq C_1/2$ and \eqref{eq:r-2} holds for $i=0,1,...,\log^8n-1$ then,
$$ \sum_{i=0}^{\log^{8}n-1}\E\big[D_{t+i}\big(2mp^{}_{d-1}-\alpha_{d-1}2m p^{}_{d-2} \big)  \big|\cH_{t+i}\big]\geq 10^{-5} \log^8 n.$$
\end{claim}
{\textbf{Proof:}} See Appendix \ref{appendix:stopsmall}.
\qed

\begin{claim}\label{claim:stoplarge}
Let $21\leq d \leq k+1$. Let $t$ be such that $\tau_d\leq t < \tau_{d-1}-\log^8n$. If $p^{t+i} \in B_d$, $p_1^{t+i}\geq C_1/2$ and \eqref{eq:r-2} holds for $i=0,1,...,\log^8n$ then,
$$ \sum_{i=0}^{\log^{8}n-1}\E\big[D_{t+i}\big(2mp^{}_{d-1}-\alpha_{d-1}2m p^{}_{d-2} \big)  \big|\cH_{t+i}\big]\geq 10^{-5}\log^8 n.$$
\end{claim}
{\textbf{Proof:}} 
$p^{t}\in B_d$  implies that $p_{d-2}^{t_1}< 1.552/(1+1.552)$ and $p_1^{t_1}\leq \frac{1}{\sum_{i=1}^{d-1}1.55^i} \leq \frac{0.55}{1.55^{d-1}-1}$. 

Thereafter \eqref{eq12}, the definition of $a_d$ (see \eqref{alphas}) and Lemma \ref{eq:sumzetas} give,
\begin{align}
 \sum_{i=0}^{\log^{8}n-1}&\E\big[D_{t+i}\big(2mp^{}_{d-1}-\alpha_{d-1} 2mp^{}_{d-2} \big)  \big|\cH_{t+i}\big]\geq 
\sum_{i=t}^{t+\log^{8}n-1}[\mathbb{I}(\z_{i}=0)+p_{d-1}^{i}- a_{d} p_{d-2}^{i}]\lambda_{d}^{i} -o(1) \nonumber
\\& \geq \sum_{i=t}^{t+\log^{8}n-1}\bigg[1-p_1^i\lambda_1^i\sum_{j=1}^d jp_j^i q_{j+1,j+1}^i-10^{-10} +o(1) -1.552 \cdot \frac{1.552}{1+{1.552}} \bigg]\lambda_{d}^i \nonumber
\\&\geq \log^8n \bigg[1-\frac{0.55d(d+1)}{1.55^{d-1}-1}-0.95\bigg] \lambda_{d} \geq 10^{-5}\log^8 n.
\label{basic1}
\end{align}
For the first inequality of the last line we used Lemma  \ref{lem:calcipq}. 
\end{proof}

We are now ready to prove the main Lemma of this subsection.

\begin{lemma}\label{lem:confine}
For $i\in\{2,3,...,k\}$ let $\sigma_i=\min\{t\geq t^*: p_i^t<\alpha_ip_{i-1}^t \text{ or }\zeta_t\geq \log^6 n \text{ or }m_t\leq n^{0.4+10^{-5}} \}$. Then w.s.h.p.\@, \begin{itemize}
    \item[(a)] $\sigma_i\leq \sigma_{i-1}$ for $i\in\{3,...,k\}$ 
    \item[(b)] $\tau_{d} \leq \sigma_d$ for $d\in\{6,...,k\}$.
\end{itemize} 
\end{lemma}
Lemma \ref{lem:confine} implies that w.s.h.p.\@, $p^t\in B_d$ for $d\in\{7,...,k+1\}$ and $t^*\leq t\leq \tau_{d-1}$. In particular, $p^{\tau_6}\in B_7$.
\begin{proof}
Recall that for $t<t'\leq \tau'$, as all the edges removed at any iteration of \kG are incident to some vertex, we have 
$$2m_{t'}-2m_{t} \leq 2(t'-t)\Delta(G)\leq 2(t'-t)\log n.$$
In addition, for any $t\leq \tau'$ and $i\in [k]$, $p^t_i-p^{t+1}_i\leq 2(i+1)\Delta(G)/2m_{t+1}\leq 2\log n/2m_{t+1}$. Thus, 
    \begin{align}\label{eq:1001}
   \Big|[2m_{t+1}& (p^{t+1}_{i}-\alpha_j p^{t+1}_{i-1})]-
    [2m_{t} (p^{t}_{i}-\alpha_j p^{t}_{i-1})]\Big|\nonumber \\&\leq
        \Big|2m_{t+1} [(p^{t+1}_{i}-\alpha_j p^{t+1}_{i-1})-(p^{t}_{i}-\alpha_j p^{t}_{i-1})]\Big| +    \Big| 2(m_{t}-m_{t+1}) (p^{t}_{i}-\alpha_j p^{t}_{i-1})]\Big| \nonumber
\\& \leq 2m_{t+1} \cdot \frac{6\log n}{2m_{t+1}} +2\log n \cdot 3 \leq 12\log n.
\end{align}
Now let $i\in\{3,...,k\}$ and assume that $\sigma_{i-1}\leq \min\{\sigma_{i-2},\sigma_{i-3},...,\sigma_2\}$. Let $t_1=\sigma_{i-1}$ and  $t_2=\max\{t^*,t_1-\log^8 n\}$.
\eqref{eq:1001} implies that
\begin{align}\label{eq:t1t2}
    [2m_{t_1} (p^{t_1}_{i-1}-\alpha_{i-1} p^{t_1}_{i-2})]-[2m_{t} (p^{t}_{i-1}-\alpha_j p^{t}_{i-2})]
    \leq 12(t-t_1)\log n 
\end{align} 
for any $t<t_1$.
Hence,
\begin{equation*}\label{eq:cond1}
    p^{t}_{i-1}-\alpha_{i-1} p^{t}_{i-2}\leq 12\log^9 n/2m_t +p^{t_1}_{i-1}-\alpha_{i-1} p^{t_1}_{i-2}\leq n^{-0.1}+0\leq n^{-0.1}<p_1^{t^*}
\end{equation*}
for $t \in[t_2,t_1]$ and by Lemma \ref{lem:initial} we have  that $t_2>t^*$. Thereafter, if $t_1=\s_{i-1} \leq \sigma_i$ then $t\in[t^*,t_1)$ implies that $p^t\in B_{i+1}\subseteq B_3$  and $p_j^t\geq C_1/2$ for $j\in [i]$ and $t\in[t_2,t_1)$ (see Lemma \ref{lem:initial}). 
Thus in the event $t_1=\s_{i-1} \leq \sigma_i$ then Lemma  \ref{lem:lowerconstrains} applies giving,
\begin{align}\label{eq:1003}
\sum_{j=0}^{\log^8n-1}\mathbb{E}[D_{t_2+j}(2m(p_{i-1}-a_{i-1}p_{i-2}))|\cH_{t_2+j}] \geq 10^{-5}2^{-k}C_1\log^8 n .  
\end{align}
In addition, 
\begin{align}\label{eq:1004}
2m_{t_1}(p_{i-1}^{t_1}-a_{i-1}p_{i-2}^{t_1})- 2m_{t_2}(p_{i-1}^{t_2}-a_{i-1}p_{i-2}^{t_2})\leq 0.  
\end{align}
Thus \eqref{eq:t1t2}, \eqref{eq:1003}, \eqref{eq:1004} and Azuma-Hoeffding inequality give,   
\begin{align*}
  \Pr(\exists i\in \{3,2,...,k\}: \sigma_{i-1}<\sigma_i) \leq \sum_{i=3}^k \sum_{t_1\leq 2m_0}\exp\set{ - \frac{(10^{-5}2^{-k}C_1\log^8 n)^2}{ 2\log^8n (12\log n)^2 }} =o(n^{-9}). 
\end{align*}
This completes the proof of part (a).

Now assume that there exists $d\in\{6,...,k\}$ such that  $\tau_{d} > \sigma_d$ and $\tau_{d'}\leq \sigma_{d'}$ for $d<d'\leq k$. Let $t_1=\sigma_d,$  $t_2=\min\{t^*,t_1-\log^8 n\}$ and $t_3=\min\{t^*,t_2-\log^{10} n\}$.
As in the proof of part (a) we have that $t_2>t_3>t^*$, $p_1^t\geq C_1/2$ and $p^{t}_{d}-\alpha_d p^{t}_{d-1}\leq 12\log^{11} n/2m_t\leq n^{-0.1}$ for $t\in [t_3,t_1)$. 
We now consider 2 cases.
\vspace{3mm}
\\\textbf{Case a:} $\tau_{d+1}\leq t_2$. Part (a) implies that $p^t\in B_d$ for $t\in[t_2,t_1)$.  Thus Lemma \ref{lem:upperconstrain} applies giving
\begin{align}\label{eq:1002}
\sum_{i=0}^{\log^{8}n-1}\E\big[D_{t_2+i}\big(2mp^{}_{d-1}-\alpha_{d-1} 2mp^{}_{d-2} \big)  \big|\cH_{t_2+i}\big]
 \geq 10^{-5}\log^{8}n.
\end{align}
In addition
\begin{align}\label{eq:10004}
2m_{t_1}(p_{i-1}^{t_1}-a_{i-1}p_{i-2}^{t_1})- 2m_{t_2}(p_{i-1}^{t_2}-a_{i-1}p_{i-2}^{t_2})\leq 0.  
\end{align}
 \eqref{eq:1001}, \eqref{eq:1002}, \eqref{eq:10004} together with Azuma-Hoeffding inequality give,
\begin{align*}
    \Pr&\bigg(\exists 6\leq  d \leq k:  \tau_d > \sigma_d \text{ and } \tau_{d+1}\leq \sigma_d-\log^8n\bigg) \nonumber
 \\&\leq \sum_{d=6}^k  \sum_{t_1\leq 2m_0}\exp\set{ - \frac{ ( 10^{-5}\log^{8}n)^2 }{2 \times \log^{8}n \times (12\log n)^2  }}  =o(n^{-9}). 
\end{align*}
\vspace{3mm}
\\\textbf{Case b:} $t_2< \tau_{d+1}$.  
\eqref{eq:t1t2} implies that $    [2m_{t_1} (p^{t_1}_{i-1}-\alpha_{i-1} p^{t_1}_{i-2})]- [2m_{t_2} (p^{t_2}_{i-1}-\alpha_j p^{t_2}_{i-2})]  \geq - 12\log^{9}n$
hence  as $ p^{t_3}_{i-1}-\alpha_j p^{t_3}_{i-2} \geq 0$,
\begin{align}\label{eq:t3}
    [2m_{t_2} (p^{t_2}_{i-1}-\alpha_{i-1} p^{t_2}_{i-2})]&- [2m_{t_3} (p^{t_3}_{i-1}-\alpha_j p^{t_3}_{i-2})]  \leq 12\log^{9}n.
\end{align}
Part (a) and $\sigma_{d+1}\geq \tau_{d+1}$ imply that $p^t\in B_{d+1}$ for $t\in[t_3,t_2]$. Thus Lemma \ref{lem:lowerconstrains} applies giving
\begin{align}\label{eq:10002}
\sum_{i=0}^{\log^{10}n-1}\E\big[D_{t_3+i}\big(2mp^{}_{d-1}-\alpha_{d-1} 2mp^{}_{d-2} \big)  \big|\cH_{t_2+i}\big]
 \geq 10^{-5}2^{-k}C_1\log^{10}n.
\end{align}

Thus \eqref{eq:1001}, \eqref{eq:10002}, \eqref{eq:t3} and Azuma-Hoeffding inequality give,   
\begin{align*}
  \Pr&(\exists d\in \{6,2,...,k\}: \tau_d> \sigma_d \text{ and } \tau_{d+1} > \sigma_d-\log^8n\bigg)) \\&\leq \sum_{i=6}^k \sum_{t_1\leq 2m_0}\exp\set{ -\Omega \bfrac{(\log^{10} n)^2}{ 2\log^{10}n (12\log n)^2 }} =o(n^{-9}). \end{align*}
\end{proof}

\subsection{Proof of Lemma \ref{lem:changeZ_t}}\label{subsection:change}
\begin{lemma}\label{01}
Conditioned on $p^{\tau_6}\in B_7$ w.s.h.p.\@ for $\tau_6\leq t<\tau'$, 
$$\text{ if $\zeta_t>0$ then } \mathbb{E}(\z_{t+1}-\z_t|\cH_t) \leq -10^{-5}.$$
\end{lemma}
\begin{proof}
See Appendix \ref{appendix:01}.
\end{proof}
{\emph{Proof of Lemma \ref{lem:changeZ_t}.}
Recall that Lemma \ref{lem:changeZ_t} asks to prove that for 
$t < \tau'$,
\begin{equation}\label{98}
    \text{  w.s.h.p.\@ if $\z_t>0$ then,
$\mathbb{E}(\z_{t+1}-\z_t|\cH_t)\leq -10^{-5}$.}
\end{equation}
We partition $[0,\tau')$ to the subintervals $I_0=[0,\tau^*]$, $I_{k+1}=(\tau^*,\tau_k]$, $I_d=(\tau_{d},\tau_{d-1}]$, $\ell=k,k-1,...,7$ 
and $I_6=[\tau_6,\tau')$. Lemmas \ref{lem:initial} and \ref{01}  imply  that \eqref{98} is satisfied for $t\in I_0\cup I_6$. Thereafter Lemma \ref{lem:confine} implies that for $d\in\{k+1,k,...,7\}$ and $t\in I_d$ we have that $p^t\in B_d$. Hence Lemmas \ref{03}, \ref{02} apply giving that \eqref{98} is satisfied $t\in \cup_{\ell=k+1}^{7} I_d$ .
\qed

\section{Reserving a set of random edges}\label{sec:reserve}
To prove Theorem \ref{thm:matchings} we use Theorem \ref{thm:kGreedy2}, a variant of Theorem \ref{thm:kGreedy}. Theorem \ref{thm:kGreedy2} states that from $G_{n,cn}^{\delta \geq k+1}$ we can remove a set of fairly random edges such that the residual graph still spans a large $k$-matching. We will later use the removed set of random edges to augment the $k$-matching into a perfect one.

\begin{theorem}\label{thm:kGreedy2}
Let $k\geq 2$,  $(k+1)/2< c=O(1)$, $n^{-0.49}\leq p =o(1)$ and $G\sim G_{n,cn}^{\d \geq k+1}$. Then, w.s.h.p.\@, there exists $V_0\subset V(G)$ of size at most $3cnp$ and $E_p\subset E(G)$ of size at least $\frac{(2cn-(k+1)n)p}{4}$ such that
\begin{itemize}
    \item[(i)] Given the set $E(G)\setminus E_p$ the edge set $E_p$ is distributed uniformly at random among all sets of size $|E_p| $ that are disjoint from $E(G)\setminus E_p$ and not incident to $V_0$ and
    \item[(ii)] $E(G)\setminus E_p$ spans a $k$-matching $M$ of size at least $kn/2-n^{0.401}$.
\end{itemize} 
In addition w.s.h.p.\@ the sets $V_0,E_p$ and $M$  can be generated in $O(n)$ time.
\end{theorem}

\begin{proof}To generate $V_0$ and $E_p$ we implement the following algorithm:

\begin{algorithm}[H]
\caption{{Generate $V_0,E_p$}}
\begin{algorithmic}[1]
\\ Let $E_0$ be a random subset of $E(G)$ where each edge belongs to $E_0$ independently with probability $p$.
\\ Reveal $E_1:=E(G)\setminus E_0$ and set $V_0 :=\{v\in V: v \text{ is incident to at most $k$ edges in }E_1\}.$
\\ Let $E^R_0$ be the set of edges in $E_0$ that are incident to $V_0$. Reveal $E_0^R$ and set $E_p:=E_0\setminus E_0^R$.
\end{algorithmic}
\end{algorithm}
\textbf{Proof of (i):} Clearly, given $V_0$, $E(G)\setminus E_p$ and $|E_p|$, $E_p$ is uniformly distributed among all the edge-sets of size $|E_p|$ that are disjoint from $E(G)\setminus E_p$ and not incident to $V_0$.
Thus it remains to show that 
 w.s.h.p.\@ $|V_0|\leq 3cnp$ and $|E_p|\geq \frac{(2cn-(k+1)n)p}{4}$.

Now,
$$\Pr(|V_0| \geq 3cnp) \leq \Pr(|E_p| \leq 3cnp/2)=\Pr(Bin(cn,p)\geq 1.5cnp)=o(n^{-9}).$$
At the line above the last equality follows from the Chernoff bound and  $p\geq n^{-0.49}$. 

To show that $|E_p|\geq \frac{(2cn-(k+1)n)p}{4}$, let $E_L$ be the subset of edges in $E(G)$ that are not incident to vertices of degree $k+1$. Then, $E_L \geq (2cn-(k+1)n)/2$.

An edge $e$ in $E_L$ belongs to $E_0\setminus E_p$ if one of its endpoints has degree $d\geq k+2$ and it is incident to at least $i\geq d-k$ edges in $E_0$ ($e$ is one of those $i$ edges). 
Thus,   
\begin{align*}
    \mathbb{E}(|E_L\cap (E_0 \setminus E_p|))&\leq n \sum_{d\geq k+2}\sum_{i=d-k}^d i \frac{e^{-\lambda}\lambda^d}{d!f_{k+1}(\lambda)} \binom{d}{i}p^{i} \leq n \sum_{d\geq k+2} \frac{e^{-\lambda}\lambda^d}{d!f_{k+1}(\lambda)} (k+2)^2p^{2} \\&\leq (k+2)^2np^{2}.
\end{align*}
In addition by changing the outcome of whether an edge belongs to $E_0$ we may decrease or increase the size of $E_L\cap (E_0 \setminus E_p)$ by at most $2\Delta(G)\leq \log n$. Hence, with $n(c,k)= 2cn-(k+1)n$, McDiarmid's inequality implies, 
$$\Pr\bigg(|E_L\cap (E_0 \setminus E_p)| \geq n(c,k)p/12 \bigg) \leq \exp\set{-\frac{\Theta(np)^2
)}{4cn\log^2n}} =o(n^{-9}). $$
 Thus, 
\begin{align*}
    \Pr\big(|E_p| \leq n(c,k)p/4\big) 
    &\leq \Pr\big(|E_L \cap E_0| \leq n(c,k)p/3 \big)+\Pr\big(E_L\cap (E_0 \setminus E_p) \geq n(c,k)p/12\big )   
    \\    & \leq \Pr\big(Bin\big(n(c,k)/2,p\big) \leq n(c,k)p/3\big)+o(n^{-9}) =o(n^{-9}).
\end{align*}

\textbf{Proof of (ii):}
We only give a brief sketch of the proof of this part as it is almost identical to the proof of Theorem \ref{thm:kGreedy}. 

Let $n_0=|V_0| \leq 3cnp=o(n).$ 
To construct the $k$-matching we implement \kG  
with the twist that every $k^{-1}(n/n_0)^{0.5}$
other steps we choose $v_t$ out of the vertices in $V_0$ that have positive degree in $G_t$ if such vertices exist.
Thus it suffices to check that in this version of the algorithm \eqref{eq:changeZ_t} is satisfied for $t$ not being a multiple of $k^{-1}(n/n_0)^{0.5}$.

For $t\leq k^{-1} n_0 (n/n_0)^{0.5}$ the set $Y_1^t$ has size at most $2t$ and therefore $p_1^t=o(1).$ \eqref{eq:change_zt} implies \eqref{eq:changeZ_t} is satisfied.

At time $t=k^{-1} n_0 (n/n_0)^{0.5}$ every vertex in $V_0$ is incident to 0 edges in $G_t$ and given $m_t$, $Y_\ell^t$, $0 \leq \ell \leq k$ and $Y_{\ell,j}^t, 0\leq \ell \leq k, 1\leq j\leq k+1$ the degree sequence of $G_t$ is uniform over the proper set of degree sequences. Thus 
we can use the exact same techniques to prove that \eqref{eq:changeZ_t} is satisfied for $k^{-1} n_0 (n/n_0)^{0.5} <t<\tau'$. 
\end{proof}

\section{Finding a $k$-factor of $G^{\d \geq k+1}_{n,m}$}\label{section:matchings}

We start with the large $k$-matching $M$ and the sets $V_0$, $E_p$ promised by Theorem \ref{thm:kGreedy2}. We then use the edges in $E_p$  to augment $M$ into a perfect $k$-matching.

To augment a given $k$-matching $M'$ we use augmenting paths. Given a $k$-matching $M'$, we say that the path $P=v_0,e_1,v_1,.....,e_s,v_s$ is  $M'$-alternating if its odd indexed edges do not belong to $M'$ whereas its even indexed edges do (here we slightly abuse the traditional definition of alternating paths where 
$E(P)\cap M$ consists either of the odd or of the even indexed edges of $P$). We say that $P$ is $M'$-augmenting if it is an $M'$-alternating path of odd length. Hence if $P$ is $M'$-augmenting then $M'\triangle E(P)$ is a $k$-matching of size $|M'|+1$.

Given a graph $G$, a vertex $r\in V(G)$, a $k$ matching $M$ of $G$ and an integer $\ell \in \mathbb{N}$ we let $T=T(r,M,\ell)$ be a random ``alternating tree'' w.r.t $M$, rooted at $r$, of height $2\ell$, that is generated as follows:

We let $L_0=\{r\}$, $E_1$ be a set of 2 random edges incident to $L_0$ that do not lie in $M$ and $L_1$ the endpoints of the edges in $E_1$ that are distinct from $r$. For $2\leq i\leq 2\ell$, given the sets of vertices (levels) $L_0=\{v\},L_1,L_2,...L_{i-1}$, if $i$ is even then we let $L_i$ be the set of vertices in $V(G) \setminus (  {\bigcup}_{j<i} L_{j})$ that are incident to $L_{i-1}$ via an edge in $M$. We let $E_i$ be the corresponding  set of edges in $M$. If $i\geq 2$ is odd then every vertex $v\in L_{i-1}$ chooses a random edge $e_v$ incident to it but not in $M$.
We let $E_i=\{e_v: e_v \not\subset \cup_{j=1}^{i-1} L_{j}\}$
and $L_{i}= \{w: \{w,v\}\in E_i \text{ for some } v\in L_{i-1}\}$.
That is $E_i$ consists of the random edges incident to $L_{i-1}$ whose other endpoint does not appear in a smaller  indexed level and $L_i$ is the set of those other endpoints. 

We also let, $L(T)=\bigcup_{0\leq i\leq \ell} L_i$, $L_{even}(T)=\bigcup_{0\leq i\leq 2\ell} L_{2i}$ and $L_{odd}(T)=L(T) \setminus L_{even}(T)$. Similarly we let $E(T)=\bigcup_{0\leq i\leq \ell} E_i$, $E_{even}(T)=\bigcup_{0\leq i\leq 2\ell} E_{2i}$ and $E_{odd}(T)=E \setminus E_{even}$.

Finally we remove $|E(T)|-(|V(T)|-1)$ edges from $E(T)$ (and their copies in $E_{even}(T)$ and $E_{odd}(T)$) with the restriction that the resultant subgraph is a tree.

Thus, at the alternating tree $T=T(r,M,\ell)$ every vertex of $V(T)$  can be reached from the root $r$ via an alternating path of length at most $2\ell$.

Given $T=T(r,M,\ell)$  we let $ext(T,T(r,M,\ell+\ell'))$ be the tree rooted at $r$ of height $2(\ell+\ell')$ whose first $2\ell$ levels are identical with the corresponding ones of $T$ and for $1\leq i\leq 2\ell'$ its $(2\ell+i)th$ level  is generated according to the same rules as the $(2\ell+i)th$ level of $T(r,M,\ell+\ell')$. Thus $ext(T,T(r,M,\ell+\ell'))$ consists of a possible extension of $T$ to an alternating tree w.r.t. $M$, rooted at $r$ and of height $2\ell+2\ell'$.   

Let $M$ be a $k$-matching and  $v,w\in V(G)$ be two vertices that are incident to at most $k-1$ edges in $M$. Let $\ell \in \mathbb{N}$, $T_v=T(v,M,\ell)$ and  $T_u=T(u,M,\ell)$. If there exists $w\in L_{odd}(T_v)$ that is incident to at most $(k-1)$ edges in $M$ then the unique $v$ to $w$ path $P_{v,w}$ defined by $T_v$ is an augmenting path. Otherwise, in the subgraph of $G$  spanned by $V(T_v)$  every vertex in an odd level of $T_v$ has degree at least $k+1\geq 3$ and every vertex at an even level has degree at least 2. In our case, this will imply that the size of $ L_{odd}(T_v) $ is either linear in $n$ or exponential in $\ell$. Finally for $v'\in L_{odd}(T_v)$ and $u'\in L_{odd}(T_u)$ if the paths $P_{v-v'}$, $P_{u'-u}$ are vertex disjoint then the path $P_{v-v'},\{v',u'\},P_{u'-u}$ is $M$-augmenting.

\begin{lemma}\label{thm:matchings_b}
Let $k\geq 2$, $G$ be a graph on $n$ vertices, $V_0\subset V(G)$, $E_p\subset \binom{V(G)}{2}$ such that,
\begin{itemize}
    \item[(i)] $G$ has minimum degree $k+1$,
    \item[(ii)] $|V_0|=O(n^{0.85})$,
    \item[(iii)] $E_p$ is an edge set of size $\Theta(n^{0.85})$ that is uniformly distributed over all subsets of $\binom{V}{2}$ of size $|E_p|$ that contain no edge incident to $V_0$ and are disjoint from $E(G)$,
    \item[(iv)] $G$ spans a $k$-matching of size at least $kn/2- n^{0.401}$,
    \item[(v)] every set $S\subset V(G)$ of size $|S|\leq k^{400}$ spans at most $|S|$ edges in $G$,
    \item[(vi)]  every set $S\subset V(G)$ of size $k^{200}\leq |S| \leq 6n^{0.999}$ spans at most $(1+10^{-10})|S|-1$ edges in $G$.
\end{itemize}
Then $(V(G),E(G)\cup E_p)\in \cP_k$ with probability $1-o(n^{-8})$. In addition, with the same probability, we can find the corresponding  $k$-factor in $O(n)$ time. 
\end{lemma}

\begin{proof}
First assume that $kn$ is even.
Let $M_0$ be a maximum $k$-matching of $G$. Condition (iv) of our Lemma implies that $|M|\geq kn/2-n^{0.401}.$ Set $Y_0=E_p$ and let $F_0$ be the set of edges in $\binom{V(G)}{2}$ that are disjoint from $V_0$ and do not lie in $E(G)$.

For $t=1,2,...,n^{0.401}$ with probability $1-o(n^{-10})$ we will construct, in $O(n^{0.59})$ time, a $k$-matching $M_t$ of size at least $\min\{  kn/2, |M_0|+t\}$, a vertex set $V_t \supseteq V_0$ of size at most $|V_0|+2tn^{0.58} \leq n^{0.99}$ and a subset $Y_t \subset Y_0$ of size at least $|Y_0|-tn^{0.25}\geq |Y_0|/2=\Theta(n^{0.85})$.
$Y_t$ will be distributed as a random subset of $F_t$ of size $|Y_t|$  where $F_t$ will be defined to be a superset of the set of edges that are disjoint from $V_t$ and not in $(Y_0\setminus Y_t) \cup E(G)$. Thus with probability $1-o(n^{-9})$, $M_{n^{0.401}}$ is a  $k$-factor of $(V(G),E(G)\cup E_p)$ that is constructed in $O(n)$ time. 

Let  $\ell_1=\log_{k}n^{0.41}$ and $\ell_2=\log_{k}n^{0.589}$. Let $0\leq t\leq n^{0.401}-1$ and suppose that we are given sets $Y_{t},V_{t},F_t$ and $M=M_{t}$ satisfying the above. If $M_t$ is a $k$-factor let $(M_{t+1},Y_{t+1},V_{t+1},F_{t+1})=(M_{t},Y_{t},V_{t},F_t)$. Otherwise let $G_t=G\cup (Y\setminus Y_t)$. W.l.o.g. we may assume that there exist distinct vertices $v,u\in V$ that are incident to at most $k-1$ edges in $M$. If no such vertices exists, then either $M$ is a $k$-factor of $G$ or there exists a vertex $v$ and an edge $\{v,u\}$ such that $v$ is incident to at most $k-2$ edges in $M$ and $\{v,u\}\notin M$. In such a case we can add $\{v,u\}$ to $M$ and remove from $M$ a different edge that is incident to $u$. 

Every vertex $w$ chooses a random set of edges $E_w'$ in $G\setminus M_t$ of size $(k+1)-d_{M_t}(w)$. For every alternating tree generated in this iteration, when needed, each vertex $w\in V$ will pick its random edges from $E_w=E_w'\cup M_t(w)$, where $M_t(w)$ is the set of edges in $M_t$ incident to $w$. We let $S_v$ be the set of vertices reachable by $u$ via a path $u=v_0,e_0,v_1,e_1,..v_s$ of length $s\leq  0.1\log_{k+1} n$ where $e_i\in E_{v_i}$ for $i=0,1,...,s-1$. Thus $|S_v| \leq \sum_{i=0}^{0.1\log_{k+1}n} (k+1)^i \leq 2n^{0.1}.$

We now implement the algorithm {\sc GenerateTree}, described below, twice. The first time with $w=v$, $S=S_v$, $S'=\emptyset$, $X=V_t$, $M=M_t$ and the second time with $w=u$, $S=S_u$, $S'=S_u'$, $X=V_t$, $M=M_v$ where the sets $S_u,S_u'$ and $M_v$ depend on the first implementation. If both of the implementations are successful then in both times the  algorithm outputs a vertex $w'$, a matching $M_{w'}$ and an alternating tree $T_{w'}$ w.r.t. $M_{w'}$ rooted at $w'$. The tree $T_{w'}$  either spans an $M_{w'}$-augmenting path or (i) its first $0.1\log_{k+1} n$ levels do not intersect $S$ (ii) there will be at least $n^{0.5889}$ many vertices in its even levels that do not belong to $V_t$ and (iii) the set of edges $M_{w'}\triangle M_w$ defines an alternating path w.r.t. $M_w$ from $w$ to $w'$ that does not intersect $S'$. We set $S_u=S_u'=\emptyset$ if $T_{v'}$ spans an $M_{v}$-augmenting path. Otherwise $S_u$ and $S_u'$  consists of the vertices in the first $0.1\log_{k} n$ and $0.02\log_{k}n$ respectively levels of $T_{v'}$.  If the algorithm does not output a tree spanning an augmenting path we will then search for an edge from  $L_{even}(T_{v'})$ to $L_{even}(T_{u'})$ in $Y_t$. Such an edge will create an augmenting path from $v'$ to $u'$.

\begin{algorithm}[H]
\caption{{\sc GenerateTree}}
\begin{algorithmic}[1]
\\Input: $w$, $M$,$X$, $S$, $S'$.
\\Set $i=1$
\While{  $i\leq n^{0.001}$ }
\\\hspace{5mm} Generate $T_w=T(w,M,\ell_1)$. Let $\bar{T_w}=ext(T_w,T(w,M,\ell_1+0.05\log_{k} n))$.
\\\hspace{5mm} Let $w'$ be a random leaf of $T_w$ and $\bar{T_{w'}}$ be the subtree of $\bar{T_w}$ rooted at $w'$.
\\\hspace{5mm} Let $P$ be the unique $w$ to $w'$ path in $T_w$.
\\\hspace{5mm} Set $M_w:=M\triangle E(P)$ 
and generate $T_{w'}=ext(\bar{T_{w'}},T(w',M_w,\ell_2))$.
\If{ for some triple $(x,M_x,T_x) \in \{(w,M,\bar{T_w}), (w',M_w,T_{w'})\}$ there exists $y\in L_{odd}(T_{x})$ that is incident to less than $k$ edges in $M_x$}{ return SUCCESS, $x$, $T_x$, $M_x$.}
\ElsIf{ (i) $|L_{even}(T_{w'})\setminus X| \geq n^{0.5889} $ and 
(ii) $\cup_{i=0}^{0.1 \log_{k}  n-1}  L(T_{w'}) \cap S= \emptyset$
\\\hspace{11mm} (iii)  $P_{w,w'}$ does not intersect $S'$}{ return SUCCESS, $v',T_{v'}$ and $M_v$.}
\EndIf 
\\\hspace{5mm} i=i+1. 
\EndWhile
\\ Return FAILURE.
\end{algorithmic}
\end{algorithm}

The proof of Lemma \ref{thm:matchings_b} (displayed after the proofs of the Claims) will follow from Claims \ref{success} and \ref{desiredsets} and the observation that {\sc{GenerateTree}} runs in $O(n^{0.59})$ time. Indeed in each of  the $O(n^{0.001})$ iterations each of the alternating trees generated consists of $O(k^{\log_k\ell_2})=O(n^{0.589})$ many vertices yielding an augmentation procedure that constructs a $k$-factor in $O(n^{0.001}\times n^{0.589} \times n^{0.401})=O(n)$ time.

Claim \ref{claim:ineq} describes the rate of growth of an alternating tree that does not span an augmenting path. We will use the corresponding rates in the proofs of Claims \ref{success} and \ref{desiredsets}. We will also make use of the special vertex $z$ appearing in the statement of Claim \ref{claim:ineq} later when proving Lemma \ref{thm:matchings_b}  for $kn$ being odd.

\begin{claim}\label{claim:ineq}
Let $M$ be a $k$-matching of $G$ and $T'=T(r,M,h)$ be an alternating tree rooted at some vertex $r$ satisfying $d_M(r)<k$, of height $2h\leq 2\log_k n^{0.999}$.  Let $z$ be any vertex of $V(G)\setminus (N(r) \cup \{r\})$ satisfying $d_M(z)=0$ and $T$ be the subtree of $T'$ that we get by pruning $T'$ at $z$. If $T$ does not span an $M$-augmenting path then, for $0\leq i \leq 390$
\begin{equation}\label{eq:smallodd}
 k^{i-1}  \leq |L_{2i+1}(T)| \leq 2k^{i}
\end{equation} 
and 
\begin{equation}\label{eq:smalleven}
 k^{i} \leq |L_{2i+2}(T)| \leq 2k^{i+1}. 
\end{equation}
In addition for $300\leq i \leq h-1$,
\begin{equation}\label{eq:even}
    (1-10^{-9}) |L_{2i}(T)|  \leq |L_{2i+1}(T)| \leq  |L_{2i}(T)|
\end{equation} 
and 
\begin{equation}\label{eq:odd}
 (1-10^{-9})k |L_{2i+1}(T)| \leq |L_{2i+2}(T)| \leq k |L_{2i+1}(T)|.
\end{equation}
\end{claim}

\begin{claim}\label{success}
If $|X|\leq n^{0.99}$, $|S|\leq 4n^{0.1}$, $|S'|\leq 4n^{0.01}$  and $S'$ does not intersect the first $0.1\log_{k} n$ levels of $T_w$ then,
$$\Pr(\text{ {\sc GenerateTree}  returns FAILURE})=o(n^{-10}).$$
\end{claim}
\begin{claim}\label{desiredsets}
Conditioned on {\sc GenerateTree} returning SUCCESS at both implementations, with probability $1-o(n^{-10})$ we can generate the desired sets $M_{t+1}$, $V_{t+1}$ and $Y_{t+1}$ in $O(n^{0.59})$ time.
\end{claim}

\noindent{\textbf{Proof of Claim \ref{claim:ineq}}}
$d_M(z)=0$, thus $T=T'\setminus\{z\}$ as every vertex in an even level is reached by an edge in $M$ and the children of a vertex $v\in L_{odd}(T')$ are connected to $v$ via an edge in $M$. Thereafter, the condition that $M$ does not span an augmenting path implies that every vertex in $L_{odd}(T')$ is incident to $k$ edges in $M$. Henceforward we write $L_{odd}(T)$ for $L_{odd}(T')\setminus\{z\}$, $L_{even}(T)$ for $L_{odd}(T')\setminus\{z\}$ and $L_{i}(T)$ for $L_{i}(T')\setminus\{z\}$.

The upper bounds  in all of the above inequalities follows from the construction of $T$. With the exception of the root which has 2 children  every vertex in $L_{2i}(T)$ has at most 1 child in $L_{2i+1}(T)$ hence $|L_{2i+1}(T)|\leq |L_{2i}(T)|$ for $i\geq 1$. Similarly every vertex in an odd level has at most $k$ children (defined by the edges in $M$)  hence $|L_{2i+2}(T)|\leq k|L_{2i+1}(T)|$ for $i\geq 0$.

Let $T_{390}$ be the subtree of $T$  spanned by its first $2\cdot 390 +1$ levels. Observe that since $d_M(z)=0$ and $r$ is not a neighbor of $z$ we have that $z$ does not appear in the first 3 levels of $T'$. Thereafter Condition (v) of Lemma \ref{thm:matchings_b} implies that exactly one of the following occurs (i) there are $2k$ vertices in $L_2(T)$, at most 2 neighbors of $z$ appear in $V(T_{390})$ and $V(T_{390})$ does not span a cycle in $G$,
(ii) there are $2k$ vertices in $L_2(T)$, at most 1 neighbor of $z$ appears in $V(T_{390})$ and $V(T_{390})$ spans at most 1 cycle in $G$,
(iii) there are at least $2k-2$ vertices in $L_2(T)$, at most 1 neighbors of $z$ appear in $V(T_{390})$ and $V(T_{390})$ spans exactly 1 cycle in $G$ which is spanned by the first 3 levels of $T_{390}$.

In all 3 cases, there exists a vertex $r'\in L_2(T)$ such that the subtree of $T_{390}$ rooted at $r'$, say $T_{r'}$, does not contain a neighbor of $z$, does not span a cycle and all the vertices in the odd (even respectively) levels  of $T_{r'}$ have $k$ (1 resp.) descendants . \eqref{eq:smallodd} and  \eqref{eq:smalleven} follows as both, the $(2i)th$ level and $(2i+1)th$ level of $T_{r'}$ consist of exactly $k^i$ vertices.

We proceed to prove that \eqref{eq:even} holds for $i\geq 390$ by induction. Assume that for $390\leq j \leq i-1$ \eqref{eq:odd} and \eqref{eq:even} hold. In addition, equations \eqref{eq:smallodd}, \eqref{eq:smalleven} hold for $i\leq 390$. 

Assume for the sake of contradiction that $(1-10^{-9}) |L_{2i}(T)| > |L_{2i+1}(T)|$ and let $S$ be the subgraph of $G$ spanned by  $ \bigcup_{j=0}^{2i+1} L_j(T) \cup\{z\}$. Every vertex in  $L_{2i}(T)$ is reached via an edge in $M$ and is incident to an edge in $E(G)\setminus M$ whose other endpoint belongs to $S$. Thus $S$ spans $N= \sum_{j=0}^{2i+1} |L_j(T)|+1$ vertices and at least $M=\sum_{j=1}^{2i+1}|L_j(T)|+(|L_{2i}(T)|-|L_{2i+1}(T))|)\geq \sum_{j=1}^{2i+1} |L_j(T)|+10^{-9}|L_{2i}(T)|+1.$ edges.

Thus,
\begin{align*}
    \frac{M}{N}&\geq 1+\frac{ 10^{-9}|L_{2i}(T)|}{4+\sum_{j=1}^{299}(|L_{2j}|+|L_{2j+1}|)+\sum_{j=300}^{i} (|L_{2j}|+|L_{2j+1}|)} 
    \\&\geq 1+\frac{ 10^{-9}|L_{2i}(T)|}{4+10k^{300}+\sum_{j=300}^{i-1} \frac{2}{[k(1-10^{-9})^2]^{i-j}}|L_{2i}|+2|L_{2i}|}
\\&\geq 1+\frac{10^{-9}|L_{2i}(T)|}{4+4k^{300}+4|L_{2i}|+2|L_{2i}|}
> 1+10^{-10}.
\end{align*}
At the second inequality we used \eqref{eq:smallodd}, \eqref{eq:smalleven} to upper bound 
$\sum_{j=1}^{299}(|L_{2j}|+|L_{2j+1}|)$ by $4k^{300}$
and the induction hypothesis to upper bound 
$(|L_{2j}|+|L_{2j+1}|)$ by $2/{[k(1-10^{-9})^2]^{i-j}}|L_{2i}|$ for $j<i$. In addition at the third inequality we used that \eqref{eq:smallodd}, \eqref{eq:smalleven} and the induction hypothesis imply that $|L_{2i}|\geq |L_{2\cdot 389}| \geq k^{388}.$

Thus $S$ is a subgraph of $G$ on $|V(S)|\leq 3+\sum_{i=1}^{h-1}4k^{i}+2k^h\leq 6k^h\leq 6n^{0.999}$ many vertices and with $|E(S)|>(1+10^{-10})|V(S)|$ edges contradicting  Condition (vi) of Lemma \ref{thm:matchings_b}. Hence \eqref{eq:even} holds. The proof of \eqref{eq:odd} follows in a similar manner as every vertex in an odd level in $T$ is incident to $k$ edges in $M$. \qed

{\textbf{Proof of Claim \ref{success}:}}
Assume that $w,M,X,S,S'$ satisfy the requirements of Claim \ref{success}.
Let $\mathcal{S}$ be the event that {\sc GenerateTree} returns success at line 8. Consider extending $\bar{T_w}$ into $Ext\bar{T}_w=ext(\bar{T_w},T(w,M,\ell_1+\ell_2))$ and let $T_{w'}'$ the subtree of $Ext\bar{T}_w$ rooted at $w'$. Then, $T_{w'}'$ is a subtree of $T_{w'}$.

For a leaf $q \in L_{2\ell_1}(T_w)$  let $T_{w,q}$ be the subtree of $Ext\bar{T}_w$ rooted at $q$. Call such a leaf \emph{good} if (I) there are more than $n^{0.5889}$ vertices in $L_{even}(T_{w,q})$ that do not lie in $X$, (II)
$(\cup_{i=0}^{0.1 \log_{k}  n-1}  L({T_{w,q}}) )\cap S= \emptyset$ and (III) $q$ is not a descendant of a vertex in $S'\cap V(T_w)$ w.r.t  the tree $T_w$.

If $\mathcal{S}$ does not occurs then Claim \ref{claim:ineq} implies that $T_w$ has at least $ [(1-10^{-9})k]^{\ell_1-1}$ and at most $2k^{\ell_1}$ leaves. For each such leaf $q$ the tree $T_{w,q}$ spans at most $4k^{\ell_2}$ vertices in its even levels while all  together span at least $[(1-10^{-9})k]^{\ell_1+\ell_2-1}$ vertices in the even levels of $Ext\bar{T}_w$ (as $\ell_1+\ell_2 \leq 0.999\log_k n$).
In addition at  most $n^{0.5889}2k^{\ell_1}$ of those vertices are spanned by trees whose root violates Condition (I). Finally  $|X| \leq n^{0.99}$. Hence in the event  $\mathcal{S}^c$ the number of leaves of $T_w$ that satisfy the Condition (I) is at least 
\begin{align*}
 \frac{ [(1-10^{-9})k]^{\ell_1+\ell_2-1}-n^{0.99}- n^{0.5889}2k^{\ell_1}  }{4k^{\ell_2}} 
&\geq  \frac{k^{\ell_1 + (\ell_1+\ell_2)\log_k(1-10^{-9})-1}}{4} -n^{0.401}-k^{0.5889\log_k n+ \ell_1-\ell_2}
\\& \geq 
\frac{k^{\ell_1-(\ell_1+\ell_2)\cdot \frac{10^{-9}}{\log k}  }}{5}-k^{ \ell_1-0.0001\log_k n} \geq \frac{k^{\ell_1-2\cdot 10^{-9}\log_k n }}{6}.
\end{align*}
Thereafter, each vertex in $S$ may belongs to the first $0.1\log_kn$ levels of a single subtree of $\bar{T_w}$ rooted at a vertex at the $(2\ell_1)th$ level of $\bar{T_w}$, hence to the first $0.1\log_kn$ levels of a single tree $T_{w,q}$. Thus at most $|S|\leq 4n^{0.1}$ leaves of $T_w$ do not satisfy Condition (II).

Finally $P_{w,w'}$ intersects  $S'$ only if $w'$ is a descendant of a vertex of $S'$ in $T_w$. $S'$ consists of at most $4n^{0.01}$ vertices none of which belongs to the first $0.1\log_{k} n$ levels of $T_w$. Hence, as $S'$ does not intersect the first $0.1\log n$ levels of $T_{w}$, $S'$ may has up to $4n^{0.01} \times2 k^{\ell_1-0.05\log_{k} n} = 2k^{\ell_1-0.04\log_{k} n}$ descendants in the $(2\ell)^{th}$ level of $T_w$.

Thus,  
\begin{align*}
    \Pr&( \mathcal{S} \text{ or }\{\text{Conditions (i),(ii),(iii) are satisfied}\})= \Pr( w' \text{ is good}) 
   \\& \geq  \frac{\frac{k^{\ell_1-2\cdot 10^{-9}\log_k n }}{6} -4n^{0.1}- k^{\ell_1-0.04\log_{k} n}}{2k^{\ell_1}} 
\\&\geq  \frac{k^{\ell_1-2\cdot 10^{-9}\log_k n }}{20k^{\ell_1}}
=  \frac{k^{-2\cdot 10^{-9}\log_k n }}{20}= \frac{n^{-2\cdot 10^{-9}}}{20}.
\end{align*}
Hence,
$$\Pr( GenerateTree \text{ returns FAILURE}) \leq  \bigg(1-\frac{n^{-2\cdot 10^{-9}}}{20} \bigg)^{n^{0.001}} \leq \exp  \bigg\{-n^{0.0001}\bigg\}=o(n^{-10}).$$
\qed

{\textbf{Proof of Claim \ref{desiredsets}:}}
If at any of the two iterations {\sc{GenerateTree}} returns a tree $T$ and a matching $M$ such that $T$ spans an $M$-augmenting path $P$ then we can let 
$M_{t+1}=M\triangle E(P)$ and $(Y_{t+1},V_{t+1},F_{t+1})=(Y_{t},V_{t},F_{t})$.

Assume that the above does not occurs and {\sc{GenerateTree}} returns SUCCESS at both implementations. Recall that at both implementations we set $X=V_t$. In addition at the  second implementation $S$ ($S'$ respectively) consist of the vertices in the first $0.1\log n$ ($0.02\log n$ resp.) levels of $T_{v'}$. Thus, since in both implementations the conditions in lines 9, 10 are satisfied, the corresponding outputs $(v',T_{v'},M_v)$ and $ (u',T_{u'},M_u)$ satisfy the following: 
\begin{itemize}
    \item [(i)] $|M_u|=|M_v|=|M|$ (by construction),
    \item[(ii)] the first $0.1\log_{k}n$ levels of $T_{v'},T_{u'}$ do not intersect (as at the second implementation, $S_u$ consists of the vertices that lie in the first $0.1\log_{k}n$ levels of $T_{v'})$),
    \item[(iii)] the $M_v$-alternating path $P_{u,u'}$ does not intersect the first $0.02\log_{k}n$ levels of $T_{v'}$ (as at the second implementation, $S_u'$ consists of the vertices that lie in the first $0.02\log_{k}n$ levels of $T_{v'}$) and 
    \item[(iv)] for $w\in\{v',u'\}$, $T_{w}$ spans in its even levels a set $B_w\subset V$ of size $n^{0.5889}$ that does not intersect $V_t$.
\end{itemize}

For $x \in \{u',v'\}$ and $z\in L_{even}(T_{x})$ denote by $P_{x,z}$ the $x-z$ alternating path defined by $T_{x}$.

Now let $D_1,D_2$ be the set of descendants of $A_1=V(P_{u,u'})$ and $A_2=\cup_{i=0}^{0.02\log_k n-1}L(T_{u'})$ in $T_{v'}$ respectively and let $B_{v'}'=B_{v'}\setminus (D_1 \cup D_2)$. Then, since $|A_1|\leq \log n$, $|A_2|\leq 4n^{0.01}$ and  $A_1$ ($A_2$ respectively) does not intersect the first $0.02\log_k n$ ($0.1\log_k n$ resp) levels of $T_{v'}$ we have that,
$$|B_{v'}'| \geq n^{0.5889}-\log_k n\times 4k^{\ell_2-0.01\log_k n}-4n^{0.01}\times 4k^{\ell_2-0.05\log_k n} \geq n^{0.588}.$$
Similarly for $v_1\in B_{v'}'$, as $v_1$ is not a descendant of $A_2$ and therefore $P_{v',v_1}$ does not intersect $A_2$ (i.e. the first $0.02\log_k n$ levels of $T_{u'}$), there exist a set $B_{u',v_1} \subset B_{u'}$ of at least $ n^{0.588}$ vertices in $B_{u'}$ that does not contain a descendant of $V(P_{v',v_1})$ in $T_{u'}$. 

Let $R_t=\{\{v_1,u_1\}:v_1\in B_{v'}',u_1\in B_{u',v_1}\}$. For $(v_1,u_1) \in R_t$ observe that $P_{v',v_1}$ is vertex disjoint from both  $P_{u,u'}$ and $P_{u',u_1}$. Thus the path $P_{v',v_1},\{v_1,u_1\},P_{u',u_1}$ is $M_u$-augmenting. Therefore, if $(Y_t\cup E(G)) \cap R_t \neq \emptyset$ and  $|Y_t\cap R_t|\leq n^{0.25}$  then we can set $V_{t+1}=V_t \cup B_{v'} \cup B_{u'}$, $Y_{t+1}=Y_t\setminus R_t$, $M_{t+1}=M_u \triangle E(P_1,\{v_1,u_1\}P_2)$ and $F_{t+1}=F_t \setminus R_t$ for some edge $\{v_1,u_1\}\in R_t$.
$|R_t|/\binom{n}{2} \geq n^{2\cdot 0.588-2}=n^{-0.824}$ and recall that $|Y_0| \geq |Y_t| \geq |Y_0|/2 = c_0n^{0.85}$ for some $c_0>0$.  Thus
\begin{align*}
    \Pr(\text{Claim} \ref{desiredsets})&\geq 
    \Pr(Bin(c_0n^{0.85}, n^{-0.824}) \in [1,n^{0.25}])
=1- o(n^{-10}).
\end{align*}
Where the last equality follows from the Chernoff Bound.
Finally, condition on $\Delta(G)\leq \log n$, the above process takes $O(n^{0.59})$ time as one has to check the edges incident to $B_{v'}$ and identify one that gives rise to an augmenting path. 
\qed

At the first application of {\sc GenerateTree} $S=S_v$ consists of at most $4n^{0.1}$ vertices and $S'=\emptyset$ hence Claim \ref{success} applies. If at its first application {\sc GenerateTree} exits at line 8 then at its second application $S=S'=\emptyset$. Else if at its first application {\sc GenerateTree} exits at line 10, then $\cup_{i=1}^{0.1\log_k n-1}L(T_{v'})\cap S_v=\emptyset$ and at the second application of {\sc GenerateTree} we set 
$S=S_u=\cup_{i=0}^{0.1\log n}L(T_{v'})$ and $S'=S_u'=\cup_{i=0}^{0.02\log n}L(T_{v'})$. Thus $|S|\leq 4n^{0.1}$, $|S'|\leq 4n^{0.01}$ and $S'$ does not intersect the first $0.1\log_kn$ levels of $T_u$. The later holds as those vertices belong to $S_v$ and $\cup_{i=1}^{0.1\log_k n-1}L(T_{v'})\cap S_v=\emptyset$. In both cases Claim \ref{success} applies. Thus the probability that either of the two implementations returns FAILURE is $o(n^{-10})$. This, together with Claim \ref{desiredsets} finish the proof of Lemma \ref{thm:matchings_b} in the case that $kn$ is even.

In the case that $kn$ is odd let $z\in V$ and $M$ be a maximum matching of $G$. Starting from $M$ we will construct a $k$-matching $M'$ of size at least $|M|-\log^4n$ such that every vertex $v$ at distance at most 3 from $z$ is incident to $k$ edges in $M'$ and $d_{M'}(z)=0$. To do so starting from $M$, remove all the edges incident to a vertex at distance at most 4 from $z$. Then, if there exists a cycle $C$ spanned by the vertices at distance at most 4 from $z$ add $E(C)$ to $M$. Condition (v) of Lemma \ref{thm:matchings_b} implies that at most one such cycle exists. The resulting edge set $M''$ is a $k$-matching satisfying $d_{M''}(z)=0$ and $|M''|\geq |M|-(\Delta(G)+1)^4 \geq |M|-\log^4 n$. Continuing from $M''$ for $i=1,2,3$, for $v$ in the $i^{th}$ neighborhood of $z$, while $d_{M''}(v)<k$ do the following: Pick an edge $e\in E(G)\setminus M''$  incident to $v$ whose other endpoint, say $v'$, is at distance $i+1$ from $z$ and does not lie on $C$. Conditions (i) and (v) of Lemma \ref{thm:matchings_b}  imply that either $v$ has at least $k$ neighbors at distance $i+1$ from $z$ or it lies on $C$ and has at least $k-1$ neighbors at distance $i+1$ from $z$. In both cases   such a vertex $v'$ exists. Finally set $M''= M''\cup e$.

Let $M'$ be the final edge set $M''$ resulted from the above procedure. Observe that $M'$ is a $k$-matching. Indeed,  every vertex $v$ at distance $i\leq 4$ was first matched to at most $m(u)\leq 2 \leq k$ vertices at distance $i-1$ from $z$ and  if $i\leq 3$ it was then further matched to $k-m(u)$ additional vertices at distance $i+1$ from $z$. Furthermore $d_{M'}(u)=d_{M}(u)\leq k$ for every vertex $u$ at distance at least 5 from $z$.

Now we can augment $M'$ to a $k$-factor of $G-\{z\}$ exactly as in the case where $kn$ is even. For that observe that for every vertex $v\in G$ we have that $d_{M_i}(v)\leq d_{M_{i+1}}(v)$ for $i\geq 0$. Thus at every iteration the initial vertices $u,v$ will be at distance at least $2$ from $z$ and Claim \ref{claim:ineq} applies.
\end{proof}
\subsection{Proofs of Theorems \ref{thm:matchings} and \ref{cor:matchings}}
{\em Proof of Theorem \ref{thm:matchings}:} Follows from Lemma \ref{thm:matchings_b} where conditions (ii)-(vi) are given by Theorem \ref{thm:kGreedy2} (with $p=n^{-0.15}$) and Lemma \ref{lem:density}.
\qed

For the proof of Theorem\ref{cor:matchings} in place of Lemma \ref{lem:density} we use Lemma \ref{lem:density2} stated below. Recall we denote by $F_0,F_1,...,F_{\binom{n}{2}}$ the random graph process, $V(F_0)=[n]$.
\begin{lemma}\label{lem:density2}
W.h.p, for $ 0\leq i \leq n\log n$  
\begin{itemize}
    \item[(i)] every set $S \subset [n]$ of size at most $k^{400}$ spans at most $|S|$ edges, 
    \item[(ii)] every set $S \subset [n]$ of size $k^{200}\leq |S| \leq 4n^{0.999}$ spans at most $(1+10^{-10})|S|$ edges,
    \item[(iii)] if $i\geq k^{99}n$ then they do not exists sets $S,T \subset [n]$,  such that $n^{0.99}\leq |S| \leq n/k^3$, $|T|\leq k|S|$, $S\cup T$ is connected and $T=N(S)$,
    \item[(iv)] if $i\geq k^{99}n$ then  the size of the $(k+1)$-core of $F_{i/10}$ is larger than $\big(1-e^{-\frac{i}{40n}}\big)n$.
\end{itemize}
\begin{proof}
If a graph $G$ has a $(k+1)$-core of size at most  
$(1-e^{-\frac{i}{40n}})n$ then there exists $S\subset [n]$ of size $|S|=e^{-i/40n}n$ such that there are less than $k$ edges from every $v\in S$ to $V(G)\setminus S$. Thus, with $r=r(s)=(1+10^{-10})s$ and $p=p(n)=\frac{n\log n}{\binom{n}{2}}$,
\begin{align*}
    \Pr&(\exists 0 \leq i\leq n\log n \text{ s.t. }F_i
    \text{ does not satisfy conditions (i)-(iv)} )
   \\& \leq \Pr( F_{n\log n}\text{ does not satisfy conditions (i),(ii)} )
   \\&+O(n^{0.5}) \sum_{i=k^{99}n}^{n\log n} \sum_{s=n^{0.99}}^{n/k^3}\sum_{t=0}^{ks} \binom{n}{s+t}\binom{s+t}{s} (s+t)^{s+t-2}\bfrac{2i}{n^2}^{s+t-1} \bigg(1-\frac{2i}{n^2}\bigg)^{(n-s-t)s}
    \\&+ 4\sum_{i=k^{99}n}^{n\log n} \binom{n}{e^{-i/40n}n}\bigg( \sum_{j=0}^k \binom{n-e^{-i/40n}n}{j}\bfrac{2(i/10)}{n^2}^j\bigg(1-\frac{2(i/10)}{n^2}\bigg)^{n-e^{-i/40n}n-j} \bigg)^{e^{- i/40n}n}
     \\ &\leq  4\sum_{s=4}^{k^{400}}n^s \binom {s^2}{s+1}p^{s+1}
    + 4\sum_{s=k^{200}}^{4n^{0.999}}\binom{n}{s}\binom{s^2}{r} p^r 
\\&  +  n^2\sum_{i=k^{99} n}^{n\log n} \sum_{s=n^{0.99}}^{n/k^3}(2en)^{(k+1)s}\bfrac{2i}{n^2}^{(k+1)s-1} e^{-is/n}  + 4\sum_{i=k^{99}n}^{n\log n} \bigg( e^{i/40n+1}  (i/n)^ke^{-0.15 i/n} \bigg)^{e^{-i/40n}n}
   \\&\leq  o(1)+ 4\sum_{s=k^{200}}^{4n^{0.999}} \bfrac{en}{s}^s\bfrac{es^2}{r}^rp^r =o(1)+ 4\sum_{s=k^{200}}^{4n^{0.999}} (enp)^s(eps)^{r-s}=o(1).    
    \end{align*}
\end{proof}

\end{lemma}

{\em Proof of Theorem\ref{cor:matchings}:} 
We will sketch the proof of Theorem \ref{cor:matchings} for the case $kn$ is even. The case $kn$ is odd can be dealt similarly to the corresponding case in Lemma \ref{thm:matchings_b}. We now consider 3 distinct intervals whose union is $\{0,1,...,n(n-1)/2\}$. In all 3 cases we conditioned on the events described by (i)-(iv) of Lemma \ref{lem:density2}.

\noindent \textbf{Case 1:} $0\leq i \leq k^{100}n$. The fact that $F_{i}^{(k+1)}$ is either empty or has order linear in $n$  and it is distributed as $G_{n,m}^{\delta \geq k+1}$ together with  Lemma \ref{thm:matchings_b} where we use Theorem \ref{thm:kGreedy2} (with $p=n^{-0.15}$) and Lemma \ref{lem:density2} to verify the underlying conditions imply that w.h.p. $F_i^{(k+1)}\in \cP_{k}$ for $0\leq i \leq k^{100}n$.

\noindent \textbf{Case 2:} $k^{100} n\leq i \leq n\log n$. We first reveal the edges of $F_{i/10}$ and then the edges of $F_i$ that are incident to vertices of $F_{i}^{(k+1)}$ that have degree less than $k+1$ in $F_{i/10}$. We let $V_1=V(F_{i/10}^{(k+1)})$ and $R$ be the set of edge in $F_i^{(k+1)}$ that have not been revealed yet. (iv) of Lemma \ref{lem:density2} implies that $|V_1| \geq (1-e^{i/40n})n$. Thereafter since the edges in $E(F_{i})\setminus E(F_{i/10})$ are uniformly distributed among all set of edges spanned by $V(F_i)$, of size $|E(F_{i})\setminus E(F_{i/10})|$, that do not intersect $E(F_{i/10})$ one can show that $R$ consists of at least $0.5i$ edges with probability $1-o(n^{-2})$.

Let $R'\subseteq R$, $v,w\in [n]$ and $M$ be a maximum $k$-matching of $F_{i}^{(k+1)}\setminus R'$ such that $d_M(v), d_M(w)\leq k-1$. In the case that $d_M(v)\leq k-2$ we may let $w=v$. Let $Q_v=Q(v,M,F_{i}^{(k+1)}\setminus R')$ be the set of vertices that are reachable from $v$ via an $M$-alternating path of even length whose first edge does not belong to $M$. Observe that if $z\in N(Q_v)$ then $z$ is incident to  some vertex in $Q_v$ via an edge in $M$ and hence $|N(Q_v)| \leq k|Q_v|$. Indeed, assume otherwise. Then there exist  $z\in N(Q_v)$ and $u\in Q_v$ such that $\{u,z\} \in E(G)\setminus M$ and $z$ does not have an $M$-neighbor in $Q_v$. The edge $\{u,z\}$ gives rise to an $M$-alternating path $P$ from $v$ to $u$ to $z$.
Now if $d_M(z)=0$ then $P$ is $M$-augmenting contradicting the maximality of $M$. Otherwise there exists some edge $\{z,z'\}\in M$. In such a case the path $P,\{z,z'\},z'$ witnesses the candidacy of $z'$ in $Q_v$ which gives a contradiction. 

By considering the $M$-alternating tree rooted at $v$, as done in the proof of Lemma \ref{thm:matchings_b}, we have that $|Q_v|\geq n^{0.99}$. Furthermore, as $|N(Q_v)| \leq k|Q_v|$, Lemma \ref{lem:density2} implies that $|Q_v\cap V_1|\geq n/k^3-e^{-i/40n}n \geq n/k^4.$ 

For every vertex $u\in Q_v$ the underlying $M$-alternating path $P_{v,u}$ from $v$ to $u$ defines a maximum $k$-matching $M_u=M\triangle E(P_{v,u})$ of $F_{i}^{(k+1)}\setminus R'$ such that $d_{M_u}(u),d_{M_u}(w)\leq k-1$. Now, by using $M_u$ in place of $M$ and $w$ in place of $v$ we can define in a similar manner the set $Q_{v,w}$ (in place of $Q_w$). This gives a set $\cT$ of at least $n^2/2k^8$ triples $(u',v',M_{u',v'})$ where $u'\in Q_v \cap V_1$, $v'\in Q_{v,w} \cap V_1$, $M_{u',v'}$ is a maximum $k$-matching of $F_i^{(k+1)}\setminus R'$ and $d_{M_{v',u'}}(v'),d_{M_{v',u'}}(u')\leq k-1$. Thus if $\{u',v'\} \in R'$ then we can augment $M_{u',v'}$ using the edge $\{u',v'\}$.

We then reveal the edges in $R'$ one by one until we reveal an edge $\{u',v'\}$ for which there exists a matching $M'$ with $(u',v',M')\in \cT$. We then use $\{u',v'\}$ to augment $M'$. We repeat this process until we either construct a $k$-factor or reveal all the edges in $R$.

The probability that the above process does not produce a perfect $k$-matching is bounded above by
\begin{align*}
 \Pr( Bin(0.5i,n^{2}/2k^{8}) \leq kn/2 )=o(n^{-2}).
\end{align*}
\noindent \textbf{Case 3:} $n\log n < i \leq \binom{n}{2}$.  Case 2 implies that w.h.p. $F_{n\log n}^{(k+1)} \in \cP_k$. Thus, since $F_i\subset F_{i+1}$ for $i\geq 0$ we have,
\begin{align*}
\Pr( \exists i\geq n\log n: F_i^{(k+1)}\notin \cP_{k}) \leq  \Pr( F_{n\log n} \notin \cP_{k})+\Pr( F_{n\log n} \neq F_{n\log n}^{(k+1)})=o(1).    
\end{align*}
\qed

\begin{appendix}
\section{Proof of Lemma \ref{lem:lambdaconvex}}\label{appendix:lambdaconvex}
Lemma \ref{lem:lambdaconvex} states,
\begin{lemma}
For $r \geq 2$ and $\lambda\geq 0$ we have that $g(r,\lambda)\geq 0$.
\end{lemma}

The inequality $a_r\geq 1.55 \geq 1.5$ and
Lemma \ref{lem:lambdamonotone} imply that it suffices to show that for $r\geq 2$, 
\begin{equation}\label{001}
    1.5\lambda_{r+1}-2.5\lambda_{r}+\lambda_{r-1}\geq 0.
\end{equation}
$\ell\leq \lambda_\ell=\lambda_\ell(\lambda) \leq \ell+\lambda$ for $\ell \in \{r-1,r,r+1\}$. Hence \eqref{001} holds for $\lambda\leq 0.2$

Let $r\geq 2$ and $\lambda \geq 0.2$. 
For $i \geq r-1$ let $q_i=\mathbb{P}(Po_{\geq r-1}(\lambda)=i)$.  Then,
$$\lambda_{r-1}= \sum_{i \geq r-1} i q_i= (r-1)q_{r-1}+rq_r+\sum_{i \geq r+1} i q_i,$$
$$\lambda_{r}=  \frac{rq_{r}}{1-q_{r-1}}+\frac{1}{1-q_{r-1}} \sum_{i \geq r+1} i q_i$$ and
$$\lambda_{r+1}=  \frac{1}{1-q_{r-1}-q_r} \sum_{i \geq r+1} i q_i.$$
Thus, 
\begin{align}
&1.5\lambda_{r+1}-2.5\lambda_{r}+\lambda_{r-1}
\\&= (r-1)q_{r-1}+rq_r- \frac{2.5rq_{r}}{1-q_{r-1}}
+\bigg(1-\frac{2.5}{1-q_{r-1}} +\frac{1.5}{1-q_{r-1}-q_r}\bigg)\sum_{i \geq r+1} i q_i 
\nonumber\\&
=\frac{(r-1)q_{r-1}(1-q_{r-1})+rq_r(1-q_{r-1})- 2.5rq_{r}}{1-q_{r-1}} 
\nonumber\\&+\bigg(\frac{(1-q_{r-1})(1-q_{r-1}-q_r)-2.5(1-q_{r-1}-q_r) +1.5(1-q_{r-1})}{(1-q_{r-1})(1-q_{r-1}-q_r)}\bigg) \sum_{i \geq r+1} i q_i 
\nonumber\\& 
=\frac{1}{1-q_{r-1}} \bigg[(r-1)q_{r-1}-1.5rq_r - (r-1)q_{r-1}^2 -rq_{r-1}q_r 
\nonumber\\&+\bigg(\frac{1.5q_r-q_{r-1}+q_{r-1}^2 +q_rq_{r-1}}{1-q_{r-1}-q_r}\bigg) \sum_{i \geq r+1} i q_i \bigg] 
\label{002}
\end{align}
Observe that,  
\begin{equation}\label{003}
\frac{\sum_{i \geq r+1} i q_i}{1-q_{r-1}-q_r} =r+1+ \frac{\sum_{i \geq r+2} (i-r-1) q_i}{1-q_{r-1}-q_k}.
\end{equation}
In addition,
\begin{align*}
    q_{r-1}=\mathbb{P}(Po_{\geq r-1}(\lambda)=r-1)\frac{\lambda^{r-1}/(r-1)!}{\sum_{i\geq r-1 } \lambda^i/(i-1)!}=
    \frac{1}{1+\sum_{i\geq 1}\frac{ \lambda^{i}}{\prod_{j=1}^i(r+j-1)}}
\end{align*}
and 
\begin{align*}
    \frac{q_r}{1-q_{r-1}}=\mathbb{P}(Po_{\geq r}(\lambda)=r)=     \frac{1}{1+\sum_{i\geq 1}\frac{ \lambda^{i}}{\prod_{j=1}^i(r+j)}}.
\end{align*}
Hence,
\begin{equation}\label{004}
q_{r-1}\leq \frac{q_r}{1-q_{r-1}}.
\end{equation}
In particular \eqref{004} implies,
\begin{equation}\label{005}
q_r\geq q_{r-1}-q_{r-1}^2.    
\end{equation}
Starting from \eqref{002}, then using \eqref{003} and then \eqref{005} we get that
\begin{align}
1.5\lambda_{r+1}&-2.5\lambda_{r}+\lambda_{r-1}
\geq \frac{1}{1-q_{r-1}} \bigg[ (r-1)q_{r-1}-1.5rq_r-(r-1)^2q_{r-1}^2-rq_{r-1}q_r \nonumber
\\&+ (r+1)(1.5q_r-q_{r-1}+q_{r-1}^2 +q_rq_{r-1})
\\&+\bigg(\frac{1.5q_r-q_{r-1}+q_{r-1}^2 +q_rq_{r-1}}{1-q_{r-1}-q_r}\bigg) \sum_{i \geq r+2} (i-r-1) q_i \bigg] 
\nonumber
\\&
= \frac{1}{1-q_{r-1}} \bigg[ 1.5q_r-2q_{r-1}+2q_{r-1}^2 +q_{r-1}q_r \nonumber
\\&+\bigg(\frac{1.5q_r-q_{r-1}+q_{r-1}^2 +q_rq_{r-1}}{1-q_{r-1}-q_r}\bigg) \sum_{i \geq r+2} (i-r-1) q_i \bigg] 
\nonumber
\\& \geq \frac{1}{1-q_{r-1}} \bigg[ -0.5 q_r+q_{r-1}q_r 
+\bigg(\frac{0.5q_r+q_rq_{r-1}}{1-q_{r-1}-q_r}\bigg) \sum_{i \geq r+2} (i-r-1) q_i \bigg]. \label{006}
\end{align}
For $r\leq 8$ and $0.2\leq \lambda \leq 5r$ we verify  that the LHS of \eqref{001} is positive computationally using Mathematica\footnote{the worksheet may be accessed at https://page.mi.fu-berlin.de/manastos/.}. If $r\leq 8$ and $\frac{\lambda}{r}\geq 5$ then $\lambda\geq 5r$ and $2q_{r+4}\geq q_{r+1}.$ Therefore  $\sum_{i \geq r+2} (i-r-1) q_i\geq \sum_{i \geq r+1} q_i=1-q_{r-1}-q_r$ and \eqref{006} implies that \eqref{001} is satisfied.

To verify \eqref{001} for $r\geq 9$ we consider the following cases:
\\{\textbf{Case 1:}} $q_{r-1} \geq 0.5$. In this case \eqref{006} implies that \eqref{001} trivially holds. 
\\{\textbf{Case 2:}} $0.3 \leq q_{r-1} \leq 0.5$. Either $\frac{\lambda}{r} \geq 0.5$ or  
$\sum_{i\geq 0}\bfrac{\lambda}{r}^i$ is finite and
\begin{align*}
    1=\sum_{i\geq r-1}q_i= q_{r-1}\sum_{i\geq r-1}\frac{\lambda^{i-(r-1)}(r-1)!}{i!} \leq q_{r-1}\sum_{i\geq r-1}\bfrac{\lambda}{r}^{i-(r-1)}=\frac{q_{r-1}}{1-\frac{\lambda}{r}}. 
\end{align*}
In both cases $\frac{\lambda}{r}\geq 0.5$. Hence for $r\geq 9$,
\begin{align*}
    0.25q_{r+1}&\leq q_{r+1}\bigg[\bfrac{\lambda}{r}^2\frac{r^2}{(r+2)(r+3)}+2\bfrac{\lambda}{r}^3\frac{r^3}{(r+2)(r+3)(r+4)}
    \\ &+3\bfrac{\lambda}{r}^4\frac{r^4}{(r+2)(r+3)(r+4)(r+5)} \bigg]=q_{r+3}+2q_{r+4}+3q_{r+5},
\end{align*}
which implies that $\sum_{i \geq r+2} (i-k-1) q_i\geq 0.25 \sum_{i \geq r+1} q_i=0.25(1-q_r-q_{r-1})$.
\eqref{001} is then implied by \eqref{006} and 
\begin{align*}
& -0.5 q_r+q_{r-1}q_r 
+(0.5q_r+q_rq_{r-1})\frac{ \sum_{i \geq r+2} (i-r-1) q_i}{1-q_{r-1}-q_{r}}\geq 0
\\& \Leftarrow -0.5 q_r+0.3q_r 
+(0.5q_r+0.3q_r)0.25\geq 0.
\end{align*}
\\{\textbf{Case 3:}} $0.1 <  q_{r-1} \leq 0.3$. As before, either $\frac{\lambda}{r} \geq 0.7$ or  
$\sum_{i\geq 0}\bfrac{\lambda}{r}^i$ is finite and $1=\frac{q_{r-1}}{1-\frac{\lambda}{r}}.$
In both cases $\frac{\lambda}{r}\geq 0.7$. Then, 
\begin{align*}
    0.75q_{r+1}&\leq q_{r+3}+2q_{r+4}+3q_{r+5},
\end{align*}
which implies that $\sum_{i \geq r+2} (i-k-1) q_i\geq 0.75 \sum_{i \geq r+1} q_i=0.75(1-q_r-q_{r-1})$.
\eqref{001} is then implied by \eqref{006} and
\begin{align*}
-0.5 q_r+0.1q_r+(0.5 q_r +0.1q_r) \cdot 0.75 \geq 0.
\end{align*}
\\{\textbf{Case 4:}} $0\leq   q_{r-1} \leq 0.1$. As before, either $\frac{\lambda}{r} \geq 0.9$ or  
$\sum_{i\geq 0}\bfrac{\lambda}{r}^i$ is finite and $1=\frac{q_{r-1}}{1-\frac{\lambda}{r}}.$
In both cases  $\frac{\lambda}{r}\geq 0.9$. Then, 
\begin{align*}
    q_{r+1}&\leq q_{r+3}+2q_{r+4}+3q_{r+5},
\end{align*}
which implies that $\sum_{i \geq r+2} (i-k-1) q_i\geq \sum_{i \geq r+1} q_i= 1-q_r-q_{r-1}$.
\eqref{001} is then implied by \eqref{006} and 
\begin{align*}
-0.5 q_r+0.5q_r \cdot 1 \geq 0.
\end{align*}

\section{Proof of Lemma \ref{lem:boundsalphasmall}}\label{appendix:boundsalphasmall}
Lemma \ref{lem:boundsalphasmall} states,
\begin{lemma}
For $2\leq r\leq 24$ we have that $1.55\leq a_{r+1}\leq 1.551$.
\end{lemma}

We will show that $1.55\leq a_{r+1}\leq 1.551$ for $2\leq r\leq 24$.
For $I\subseteq \mathbb{R}_{\geq 0}$ let 
$$a_{r,I}'=\underset{\lambda \in I}{\sup} \bigg\{
a_r-\frac{g(r,\lambda)}{\lambda_{r+1}}
+\frac{0.55\lambda_1}{(1.55^{\max\{r,3\}}-1)\lambda_{r+1}}  \bigg[ (r+1) q_{r+1,r+1} - r q_{r,r}\bigg]
 \bigg\}.$$
Recall that, $a_2=1.55$ and for $2\leq r\leq 24$,
\begin{equation*}\label{alphas2}
a_{r+1}=\max\{a_{r},a_{r,\mathbb{R}_{\geq 0}}'\}+10^{-5}.
\end{equation*} 
Thus it suffices to show that for $2\leq r\leq 24$,
\begin{equation}\label{100}
a_{r,\mathbb{R}_{\geq 0}}' - a_r\leq 4\cdot 10^{-5}.
\end{equation}
Let $I_{1,r}=[0,0.12r),$ $I_{2,r}=[0.12r,16+r]$ and $I_{3,r}=(16+r,\infty).$

Lemmas \ref{lem:lambdamonotone} and \ref{lem:lambdaconvex} imply that for $2\leq r\leq 24$,
\begin{align}
    a_{r,I_3}'-a_r&\leq  \underset{\lambda \in I_3}{\sup} \bigg\{
\frac{0.55}{(1.55^{\max\{r,3\}}-1)}  (r+1) q_{r+1,r+1} 
 \bigg\} \nonumber
 \\&\leq   \underset{\lambda \geq 16+r,2\leq  r\leq 24}{\sup} \bigg\{
\frac{0.55(r+1) \frac{\lambda^{r+1}/(r+1)!}{\lambda^{16+r}/(16+r)!}}{1.55^{\max\{r,3\}}-1}  
 \bigg\} 
 \leq  \underset{2\leq  r\leq 24}{\max} \bigg\{
\frac{0.55 \frac{(16+r)^{r+1}/r!}{(16+r)^{16+r}/(16+r)!}}{1.55^{\max\{r,3\}}-1}  
 \bigg\} \nonumber
\\& \leq 4\cdot 10^{-5}. \label{citeMathematica}
\end{align}
Thereafter for $\lambda \leq 0.12r$ and $2\leq r\leq 24$
\begin{align*}
    q_{r+1,r+1} &= \frac{(r+1)\lambda^{r+1}/(r+1)!}{\sum_{i\geq r+1} i\lambda^{i}/i!}
    = \frac{1}{\sum_{i\geq r+1} r!\lambda^{i-(r+1)}/(i-1)!} \leq      \frac{1}{1+ \frac{\lambda}{r}}
\leq 1-\frac{0.89\lambda}{r}.
\end{align*}
On the other hand,
\begin{align*}
    q_{r,r} &    = \frac{1}{1+\sum_{i\geq r+1} (r-1)!\lambda^{i-r}/(i-1)!}
    \\&\geq 1-\sum_{i\geq r+1} (r-1)!\lambda^{i-r}/(i-1)! \geq 1-\sum_{i\geq 1} \bfrac{\lambda}{r}^i \geq  1-\frac{\lambda}{r}\cdot \frac{1}{1-\lambda/r} \geq 1-\frac{1.14\lambda}{r}.
\end{align*}
In addition for $\lambda\leq 0.12r$,
\begin{align*}
    \lambda_r&=\frac{\sum_{i\geq r} i\lambda^i/i!}{\sum_{i\geq r} \lambda^i/i!} = r+ \frac{\sum_{i\geq r+1} (i-r)\lambda^i/i!}{\sum_{i\geq r} \lambda^{i}/i!}\leq r +\sum_{i\geq 1}\bfrac{\lambda}{r+1}^i,
\end{align*}
\begin{align*}
    \lambda_{r-1}&={r-1}+ \frac{\sum_{i\geq r} (i-(r-1))\lambda^i/i!}{\sum_{i\geq r-1} \lambda^{i}/i!}
    \geq {r-1}+ \frac{ \lambda^r/r!}{\sum_{i\geq r-1} \lambda^{i}/i!}  \geq r-1 +\frac{\l/r}{\sum_{i\geq 0} \bfrac{\l}{r}^i}
\\&    \geq r-1 +(1-\lambda/r)(\lambda/r) 
\end{align*}
and
\begin{align*}
    \lambda_{r+1}
    \geq r+1 +(1-\lambda/(r+2))(\lambda/(r+2))\geq (r+1).
\end{align*}
Therefore for $r\geq 2$ and $\lambda\leq 0.12r$,
\begin{align*}
    -(1.5\lambda_{r+1}-2.5\l_r+\l_{r-1})&\leq-0.5 -1.5\frac{\l}{r+2}\bigg(1-\frac{\lambda}{r+2}\bigg) +2.5 \sum_{i\geq 1}\bfrac{\lambda}{r+1}^i  -\frac{\l}{r}\bigg(1-\frac{\lambda}{r} \bigg)
\\& \leq-0.5 -\frac{(2-0.5r)\l}{r(r+1)(r+2)} +2.5 \sum_{i\geq 2}\bfrac{\lambda}{r+1}^i  
+\bfrac{\lambda}{r+2}^2+\bfrac{\lambda}{r}^2
\\& \leq -0.5+\frac{\l}{7 r} +2.5 \bfrac{\l}{r} \sum_{i\geq 0}\bfrac{\lambda}{r}^i  
+2\bfrac{\lambda}{r}^2
\\& \leq -0.5+ \frac{\l}{7 r} +2.5 \bfrac{\l}{r} \sum_{i\geq 0}\bfrac{\lambda}{r}^i  
+2\bfrac{\lambda}{r}^2
\\& \leq -0.5+\frac{0.12}{7}+\frac{0.12^2}{0.88}+2\cdot 0.12^2\leq -0.43.
\end{align*}
Combining the above  we get that for $2\leq r\leq 24$ and $\l\leq 0.12r$,
\begin{align*}
    a_{r,I_{1,r}}'-a_r&\leq  \underset{\lambda \in I_{1,r}}{\sup} \bigg\{-\frac{1.5 \lambda_{r+1}-2.5\lambda_{r}+\lambda_{r-1}}{\lambda_{r+1}}+
\frac{\lambda_1\cdot [(r+1) q_{r+1,r+1} -rq_{r,r}] }{2(1.5^{\max\{r,3\}}-1)\lambda_{r+1}}  \bigg\}
\\ &\leq  \underset{\lambda \in I_{1,r}}{\sup} \bigg\{\frac{-0.43}{\lambda_{r+1}} + \frac{(1+0.12r)\cdot [(r+1)(1-0.89\lambda/r) -r(1-1.14\lambda/r)]}{2\lambda_{r+1}(1.5^{\max\{r,3\}}-1)}
 \bigg\} 
\\&\leq \frac{1}{\lambda_{r+1}}\bigg[-0.43+
\frac{(1+0.12r)(1+0.25\cdot r\cdot  \lambda/r)}{2(1.5^{\max\{r,3\}}-1) }\bigg]
 \leq   0.
\end{align*}
At the very last inequality we used that $\lambda/r \leq 0.12$.

Thus to verify \eqref{100} it suffices to show that for $2\leq r\leq 24$,
\begin{equation}\label{150}
a_{r,I_{2,r}}' -a_r\leq 4\cdot 10^{-5}.
\end{equation}
We verify \eqref{150} computationally, using Mathematica.

\section{Proof of Lemma \ref{02}}\label{appendix:02}
Lemma \ref{02} states,
\begin{lemma}
For $3\leq d \leq 10$ and $\tau_{d} \leq t \leq \tau_{d-1}-1$ if $p^t \in B_d$ then \eqref{eq:changeZ_t} holds.
\end{lemma}

We have to show that for $3\leq d \leq \min\{10,k+1\}$ and $\tau_{d} \leq t < \tau_{ d-1}$, if  $\z_t>0$ and $p^t\in B_d$ then     $\mathbb{E}(D_t(\z)|\cH_t)\leq -0.001$. For $\tau_d\leq t< \tau_{d-1}$ we have that $\sum_{i=1}^{d}p_i=1+o(1)$. Hence \eqref{eq:echange_zt} implies that if 
$\tau_{d} \leq t < \tau_{\leq d-1}$  holds $\z_t>0$ then (with $p_d^*=1-\sum_{i=1}^{d-1}1.55^{i-1}p_1$), 
\begin{align}
    \mathbb{E}(D_t(\z)|\cH_t)&\leq -1+o(1)+\sup_{x\geq 0} \bigg\{  p_1 \cdot \lambda_1(x) \cdot
    \sum_{i=1}^{d} i p_i q_{i+1,i+1} \bigg\}\nonumber\\
&\leq -1+o(1)+\sup_{x\geq 0} \bigg\{  p_1 \cdot \lambda_1(x) q_{d+1,d+1}\cdot\bigg(dp_d^* +    \sum_{i=1}^{d-1} i 1.55^{i-1}p_1 \bigg) \bigg\}    \leq 10^{-5}\label{qq1}
\end{align}
At the second inequality we used that $q_{i+1,i+1}=\frac{(i+1)x^{i+1}/(i+1)!}{\sum_{j\geq i+1}jx^{j}/j!}=\frac{1}{\sum_{j\geq i}(i!/j!)x^{j-i}}$ is increasing w.r.t. $i$ and  that $p^t\in B_d$ implies that $p_i\geq 1.55^{i-1}p_1$ for $i\leq d-1$.  
We verify \eqref{qq1} computationally using Mathematica.

\section{Proof of Lemma \ref{lem:initial}}\label{appendix:initial}
Recall, $t^*=\bfrac{1}{40ck}^4n$. Lemma \ref{lem:initial} states,
\begin{lemma}
Let $C_1=p_1^{t^*}$. W.s.h.p.\@,
\begin{itemize}
\item[(i)]  $p_1^t \lambda_1^t \sum_{\ell=1}^k \ell p_\ell^tq_{\ell+1,\ell+1}^t\leq 0.5$ for $t\leq t^*$,
\item[(ii)] $C_1=\Omega(1)$ and $p^{t^*}_j-a_j  p^{t^*}_{j-1}\geq C_1$ for $2\leq j\leq k$,
\item[(iii)] let $\sigma^*=\min\{\tau_3, \min\{t\geq t^*: p^t \notin B_3\}\}$ then $p_1^t\geq \min\{0.04,C_1/2\}$ for $t^*\leq t\leq \sigma^*$.
\end{itemize}
\end{lemma}

\begin{proof} 
 For $j\in \{0,1,...,k\}$ and $0\leq t\leq t^*$ let $X_j^t=\cup_{0\leq t'\leq t}Y_{j}^{t'}$ and $Z_j^t$ be the subset of vertices of $X_j^{t}$ that either belong to $X_0^t$ or are incident to a vertex in $X_0^t$.  Then $X_0^t\subseteq X_1^t\subseteq X_2^t\subseteq...\subseteq X_{k}^t$ and $|X_{k}^t| \leq 2t^*$. Let $\lambda\leq 2c$ be defined by \eqref{eq:lambda}.

For $i \geq  20c \geq 10(k+1) $, let $\cE_i$ be the event that  there do not exist more than $ 2^{20c-i} n$ vertices of degree $i$ in $G$. Also let $\cE=\cap_{i\geq  20c } \cE_i$. Equation \eqref{eq:models} implies,
\begin{align*}
    \Pr(\neg \cE)&\leq \sum_{i \geq   20c  }O(n^{0.5}) \Pr(\neg \cE_i)
\leq O(n^{0.5})\sum_{i \geq   20c } \Pr\bigg(Bin\bigg(n, \frac{ \l^{i}/i!}{ \l^{k+1}/(k+1)!} \bigg) > 2^{  20c -i}n \bigg)
    \\ & \leq O(n^{0.5})\sum_{i \geq  20c }  \Pr\bigg(Bin\bigg(n, \bfrac{e\l}{i}^{i-k-1} \bigg)> 2^{  20c -i}n             \bigg) 
        \\ & \leq O(n^{0.5}) \sum_{i \geq  20 c}  \Pr\bigg(Bin\bigg(n, \bfrac{2ce}{20c}^{i-k-1} \bigg)> 2^{  20c -i}n             \bigg)=o(n^{-9}). 
\end{align*}
At the last equality we used the Chernoff Bound (see Theorem \ref{thm:chernoff}).

Thus, in the event $\cE$, the number of edges incident to a set of vertices of size $s$ is at most $ (20c+1+\log_2(n/s)) s.$

In particular for $t\leq t^*$, since every vertex that does not belong to $Y_{k+1}^{t}$ is incident to a vertex in $\{w_t,v_t:t\leq t^*\}$,
\begin{equation}\label{yk}
|Y_{k+1}^{t}| \geq n- (20c+2+\log_2n -\log_2(2t^*))2t^*
\geq n - \frac{[20c+2+4\log_2(40ck)]n}{(40ck)^4}  \geq (1-(101k^2)^{-1})n.
\end{equation}
Similarly,
\begin{equation}\label{mt}
2m_0\geq 2m_t\geq (1-(101k^2)^{-1})2m_0\geq 0.995(2m_0) \text{ for }t\leq t^*
\end{equation}
and
\begin{equation}\label{pl}
   p_1^t\leq \frac{|Y_1^t|\lambda_1^t}{|Y_{k+1}^t|\lambda_{k+1}^t}+o(1)\leq (100k^2)^{-1} \text{ for }t\leq t^*.
\end{equation}

For the rest of the proof we condition on the event $\cE$.

\emph{(i)} Similarly to the derivation of \eqref{eq:ipq}  (using \eqref{pl} in place of Lemma \ref{lem:p1}), we have,
$$p_1^t \lambda_1^t \sum_{\ell=1}^{k+1} \ell p_\ell^tq_{\ell+1,\ell+1}^t\leq (100k^2)^{-1} \times (k+1)(k+2)\leq 0.5.$$

\emph{(ii)} 
First observe that the size of $X^t_j$ is increasing w.r.t. $t$. Thereafter, for $0\leq t \leq t^*$, the size of $X^t_j$ increases by 1 only if $v_{t}$ or $w_{t}$ belongs to $Y_{j+1}^t\subseteq X_{j+1}^t$.
Thus, for $j\in[k]$,
\begin{align*}
    \mathbb{E}(D_t(|X_j|)|\cH_t) \leq \frac{2|X_j^t|\lambda_j}{|Y_{k+1}^t|\lambda_{k+1}} \leq \frac{2.01|X_j^t|}{n}.
\end{align*} 
At the last inequality we used \eqref{yk} and Lemma \ref{lem:lambdamonotone}.

As  $\Delta(G)\leq \log n$ (see Lemma \ref{lem:degrees}), a standard martingale argument implies that w.s.h.p.\@, 
\begin{align}\label{eq:78}
|X^{t^*}_j| \leq  \frac{2.02t^*}{n}|X_{j+1}^{t^*}| +o(n) \leq 0.01|X_{j+1}^{t^*}|+o(n) \text{ for }j\in[k]. 
\end{align}

Similarly, for $j\in [k]$,
\begin{align*}
    \mathbb{E}(D(|Z_j|))|\cH_t) &\leq p_1^t (\lambda_1^t+1) \leq \frac{\lambda_1^t|X_1^t|}{n} (\lambda_1^t+1)+o(1).
\end{align*} 
Now for $t\leq t^*$, \eqref{yk} and \eqref{mt} imply,
$$ \lambda_1\leq \l_{k+1}\leq \frac{2m_t}{|Y_{k+1}^t|} \leq 
\frac{2m_0}{|Y_{k+1}^{t^*}|}\leq 3c.$$
As  $\Delta(G)\leq \log n$ a standard martingale argument implies that for $t\leq t^*$, w.s.h.p.\@, 
\begin{align}\label{eq:79}
|Z^{t^*}_j| \leq  \frac{3c(3c+1)t^*}{n}   |X_1^{t^*}|+o(n)\leq  \frac{3c(3c+1)}{(40ck)^4}|X_1^{t^*}|   +o (n)\leq 0.01|X_1^t| +o(n) \text{ for }j\in[k]. 
\end{align}

On the other hand, a vertex $v$ belongs to $X_j^t\cup Z_j^t$ if there exists $t_1<t_2<...<t_{k-j}\leq t^*$ such that $v_{t_1}= v$ or $w_{t_1}= v$ and   $w_{t_i}=v$  for $i=2,...,k-j$ and $v$ is not incident to $Z_t^{t^*}$. Thus, equations \eqref{mt} and \eqref{yk} imply, 
\begin{align*}
    \mathbb{E}(|X_j^t\cup Z_j^t|)\geq 0.99n\binom{t^*}{k-j}  \sum_{d\geq k+1} \bfrac{1}{2m_{t^*}}^{k-j}=\Omega(n). 
\end{align*} 

A standard martingale argument implies that w.s.h.p.\@
\begin{align}\label{eq:80}
|X_j^t\cup Z_j^t| =\Omega(n) \text{ for }j\in [k] \text{ and } t\leq t^*.
\end{align}

Equations \eqref{eq:79} and \eqref{eq:80} imply that
\begin{align*}
    |Y_1^{t^*}|\geq |X_1^{t^*}|-|Z_1^{t^*}|\geq 0.9|X_1^{t^*}|=\Omega(n).
\end{align*}
Thus $p_1^{t^*}=\Omega(1).$

Similarly, equations \eqref{eq:78}, \eqref{eq:79} and \eqref{eq:80} imply that 
\begin{align*}
    |Y_\ell^{t^*}|- 2.6|Y_{\ell-1}^{t^*}| \geq |X_\ell^{t^*}|- 2.6|X_{\ell-1}^{t^*}|-|Z_\ell^{t^*}|\geq 0.5|X_\ell^{t^*}| =\Omega(n).
\end{align*}

Thus $p_\ell^{t^*}\geq 2.6p_{\ell-1}^{t^*}$ for $\ell\in\{2,3,...,k+1\}$.
Hence,
$p_\ell^{t^*}\geq 1.6p_{\ell-1}^{t^*}+p_1^{t^*}$ for $\ell\in\{2,3,...,k+1\}$.
\vspace{3mm}
\\{\emph(iii)} For $t\leq \tau_3$ if  $p^t\in B_3$ and $ p_1^t\leq 0.045$ then,
\begin{align*}
    \mathbb{E}(|Y_1^{t+1}|- |Y_1^{t}| \big| \cH_t)
    &=\lambda_2 p_2^t-(\lambda_1+1)p_1^t-p_1^t\lambda_1p_1^tq_{2,2} +o(1)
    \\ &\geq p_1^t\bigg[ 1.55 \lambda_2 -(\lambda_1+1)-p_1^t \lambda_1 q_{2,2} \bigg] +o(1)
    \\ &\geq p_1^t\bigg[ 0.05\lambda_2 -p_1^{t*} \lambda_1 \bigg] +o(1) \geq 0.02p_1^t \geq 0.01p_1^{t^*}.
\end{align*}
At the second inequality we used that $1.5\lambda_2\geq \lambda 1+0.5\lambda_2\geq \lambda_1+1$.
A standard martingale argument implies that $p_1^t\geq \min\{0.5p_1^{t*},0.04\}$
for $t^*\leq t\leq \sigma^*$.
\end{proof}

\section{Proof of Claim \ref{claim:stopsmall}}\label{appendix:stopsmall}
\begin{claim}\label{claim:stopsmall1}
Let $7\leq d \leq 20$. Let $t$ be such that $\tau_d\leq t  < \tau_{d-1}-\log^8 n$. If $p^{t+i} \in B_d$, $p_1^{t+i}\geq C_1/2$ and \eqref{eq:r-2} holds for $i=0,1,...,\log^8n$ then,
$$ \sum_{i=0}^{\log^{8}n-1}\E\big[D_{t+i}\big(2mp^{}_{d-1}-\alpha_{d-1}2m p^{}_{d-2} \big)  \big|\cH_{t+i}\big]>10^{-5}\log^8 n.$$
\end{claim}

Comparing to the proof of Claim \ref{claim:stoplarge}
it suffices to show that for $7\leq d \leq 20$ and $\tau_{d} \leq t <\tau_{d-1 }$   if $\alpha_{d-1}  p_{d-2} \leq p_{d-1}\leq \alpha_{d-1}  p_{d-2}+n^{-0.1}$ and $p^t \in B_d$ then the following inequality holds:
\begin{align}
&\bigg(1-p_1\lambda_1\sum_{i=1}^{d}ip_iq_{i+1,i+1}+p_{d}-a_{d}p_{d-1}\bigg)\lambda_{d}  \nonumber
\\&+ \bigg\{(a_{d}-a_{d-1}) \lambda_{d} +g(d-1,\lambda)  -  p_1 \lambda_1 \bigg[ d q_{d,d} - (d-1) q_{d-1,d-1}\bigg]\bigg\} p_{d-1}>10^{-5}.\label{200}
\end{align}

By the definition of $\alpha_{\ell}$ the second line of \eqref{200} is positive. In addition $q_{i,i}$ is increasing with respect to $i$. 
Moreover, using Mathematica we verify that, for $7\leq d\leq 20$ we have $p_1\lambda_1dq_{d+1,d+1}<1$.
Thus for $7\leq d \leq 20$ and $\tau_{d} \leq t <\tau_{d-1 }$, if $p_{d-1}\leq \alpha_{d-1} p_{d-2}+n^{-0.1}$ and $p^t \in B_d$ then using $1.55=a_2<a_i < 1.552$ and $p_1\leq \frac{0.55}{1.55^{d-1}-1}$ we have that the LHS of \eqref{200}, normalized by $\lambda_d$, is bounded below by 
\begin{align}
opt_{d} =\inf_{0\leq p_1 \leq\frac{0.55}{ 1.55^{d-1} -1}}
\inf_{\lambda \geq 0}\bigg\{ &1 -a_{d-1}  p_{d-1} -p_1\lambda_1\sum_{i=1}^{d-1}ip_i q_{i+1,i+1}-o(1)
:\nonumber
\\&p_i =1.55^{i-1}p_1 \text{ for } 1\leq i\leq d-3, \nonumber
\\&  1.55p_{d-3} \leq p_{d-2},  \nonumber
\\&   1.55 p_{d-2}\leq \a_{d-1}p_{d-2}=p_{d-1} \leq 1.552 p_{d-2} \text{ and } \sum_{i=1}^{d-1}p_i =1. \bigg\}  \nonumber
\\\geq \inf_{0\leq p_1 \leq\frac{0.55}{ 1.55^{d-1} -1}}
\inf_{\lambda \geq 0}\bigg\{ &1-1.552\frac{1.552 \beta}{2.552} -p_1\lambda_1\sum_{i=1}^{d-3}i1.55^{i-1}p_1 q_{i+1,i+1} \label{lp} \\&
-p_1\lambda_1\frac{(d-2) \beta}{2.552}q_{d-1,d-1} 
-p_1\lambda_1\frac{(d-1)1.552 \beta}{2.552}q_{d,d}
- o(1):\nonumber \\&
\beta= 1-\sum_{i=1}^{d-3}1.55^{i-1}p_1\bigg\}. \nonumber
\end{align}

 We verify that the objective of \eqref{lp} is strictly larger than $10^{-5}$ for $7\leq d\leq 20$
using Mathematica.

\section{Proof of Lemma \ref{01}}\label{appendix:01}
Lemma \ref{01} states,
\begin{lemma}\label{011}
Conditioned on $p^{\tau_6}\in B_7$ w.s.h.p.\@ for $\tau_6\leq t<\tau'$, 
\begin{align}\label{eq:app}
\text{  if $\zeta_t>0$ then, }    \mathbb{E}(\z_{t+1}-\z_t|\cH_t) \leq -10^{-5}.
\end{align}
\end{lemma}
The proof of Lemma \ref{011} will follow from Lemmas \ref{0116}, \ref{0115}, \ref{0114} and \ref{0113} found in the following subsections. Each of these Lemmas studies the process during some interval of the form $[\tau_i,\tau_{i-1})$, $i\in \{6,5,4,3\}$. Lemmas \ref{0116}, \ref{0115} and \ref{02} imply that
 with probability $1-o(n^{-9})$ for $\tau_6\leq t\leq \tau_4$ \eqref{eq:app} holds.

In the following subsections we condition on Lemma \ref{lem:confine}. Hence $\sigma_6\leq \sigma_5 \leq \sigma_4\leq \sigma_3 \leq \sigma_2$. In addition recall, that as it has been shown at the proof of Lemma \ref{lem:changeZ_t}, if \eqref{eq:app} holds for $t<\tau_3$ then $\tau_3=\tau'$.

\subsection{The process during the interval $[\tau_6,\tau_5]$}

\begin{lemma}\label{0116}
Conditioned on $p^{\tau_6}\in B_7$ w.s.h.p.\@ $p^t\in B_6$ for $\tau_6\leq t \leq \tau_5$. \end{lemma}
\begin{proof}(Sketch) 
Let $C^+_6:=\{x\in \mathbb{R}^{k+1}: x_5 \geq \a_5x_4^t+0.05x_6^t \}$ and $C^-_6:=\{x\in \mathbb{R}^{k+1}: x_5 \leq 1.56x_4^t+0.05x_6^t \}$.
Let $B_6^+=B_6\cap C_6^+$ and $B_6^-=B_6\cap C_6^-$. We consider 2 cases depending whether there exists $\tau_6 \leq t'<\sigma_5$ such that $p^{t'}\notin C^-_6$. If such $t'$ exists then $p^{t'}$ lies inside $B_6^+$ and away from its boundary (recall that $\a_5\leq 1.552$). Thereafter similar arguments as the ones used for the proof of Lemma \ref{lem:confine} imply that $p^t\in B_6^+\subset B_6$ for $t\in[t',\tau_5]$. Otherwise, we have that $p_5^t\leq 1.56p_4^t+0.05p^t_6$ for $t\leq \sigma_5$. In such a case we show that the quantity $p_5^t-a_5 p_4^t$, which at time $\tau_6$ is non-negative, increases, passing at some point the quantity $0.01p_4^t+0.05p_6^t$ hence giving a contradiction.

{\textbf{Case 1}:} There exists $\tau_6 \leq t'<\sigma_5$ such that $p^{t'}\notin C^-_6$.

In Case 1 we will prove that w.s.h.p.\@ $p^t\in B_6^+$ for $t'\leq t \leq \tau_5$. For this, as $B_6^+\subset B_6\setminus C_6^- \subset B_6$, comparing to the proof of Lemma \ref{lem:confine}  it suffices to show that if $p^t\in B_6$ and $p_5^t=1.56p_4^t +0.05p_6^t+o(1)$ then 
\begin{align}
    (1+p_6-p_1\lambda_1\sum_{i=1}^6ip_iq_{i+1,i+1})(1.05\lambda_6+0.05)-(2.56\lambda_5+1)p_5+1.56(\lambda_4+1)p_4 \nonumber
    \\ +p_1\lambda_1(-6p_5q_{5,5}+1.56\cdot 5p_4q_{5,5}+0.05\cdot7p_6q_{7,7})>0. \label{eq:61}
\end{align}
Similarly to \eqref{lp} we have that for $t\geq \tau_6$ if $p\in B_6$ then the RHS of \eqref{eq:61} is minimized when $p_i=a_ip_{i-1}$ for $i=2,3$. We verify \eqref{eq:61} using Mathematica.

{\textbf{Case 2}:} $p^t\in C^-_6$ for  $\tau_6 \leq t<\sigma_5$.

$p^{\tau_6}\in B_7$ implies that $p_6^{\tau_6}\geq \frac{0.55*1.55^5}{1.55^6-1}>0.38$. Let $t'=\min\{t>\tau_6:p_6^t\leq 0.2\}$. One can verify using Mathematica that while $p^t\in B_6^-$ and $\tau_6\leq t<t'$,  both quantities
$\mathbb{E}(D_t(2m(p_5-a_5p_4))$ and $\mathbb{E}(D_t(2m(p_5-a_5p_4+0.1p_6))$, with $\mathbb{I}(\z_t=0)$ substituted by $1-p_1^t\l_1^t\sum_{i=1}^6ip_1q_{i,i}$ (see Lemma \ref{eq:sumzetas}),  are strictly positive. 
Hence w.s.h.p.\@ we have that $p^t\in B_6$ for $t\in[\tau_6,t']$ and either $t'\geq \tau_5$ or $t'<\tau_5$ and $p^{t'} \in B_6$. In the later case, from a standard martingale argument, it follows that 
\begin{align*}
    2m_{t'}(p_5^{t'}-\a_5p_4^{t'}) &\geq 2m_{\tau_6}(p_5^{\tau_6}-\a_5p_4^{\tau_6})+
    0.1(2m_{\tau_6}p_6^{\tau_6}-2m_{t'}p_6^{t'})+o(m_{\tau_6}-m_{t'})
    \\&\geq 0+0.1(2m_{\tau_6}\cdot 0.38-2m_{t'} \cdot 0.2) +o(m_{\tau_6}-m_{t'}) \geq 0.18 \cdot 0.1 \cdot 2m_{t'}.
\end{align*}
Hence,
\begin{align*}
    p_5^{t'}-\a_5p_4^{t'} \geq 0.18 \cdot 0.1 
    \geq 0.01\cdot \frac{1}{2.55}+0.01
    \geq 0.01\cdot \frac{p_4^{t'}}{p_4^{t'}+p_5^{t'}}+0.01\geq  0.01p_4^{t'}+0.05p_6^{t'},
\end{align*}
which gives a contradiction. Thus in Case 2 we have that $t'=\tau_5$ w.s.h.p.\@.

In both, Case 1 and Case 2, we have that w.s.h.p.\@ $p^t\in B_6$ for $t\in [\tau_6,\tau_5]$.
\end{proof}

\subsection{The process during the interval $[\tau_5,\tau_4]$}

\begin{lemma}\label{0115}
Conditioned on $p^{\tau_5}\in B_6$ w.s.h.p.\@  $p^t\in B_5$ for $\tau_5\leq t \leq \tau_4$. \end{lemma}

\begin{proof}(Sketch)  
Let $C^+_5:=\{x\in \mathbb{R}^k: x_4 \geq \a_4x_3^t+0.25x_5^t \}$ and $C^-_5:=\{x\in \mathbb{R}^k: x_4 \leq 1.56p_3^t+0.05p_5^t \}$.
Let $B_5^+=B_5\cap C_5^+$ and $B_5^-=B_5\cap C_5^-$.

{\textbf{Case 1}:} There exists $\tau_5 \leq t'<\sigma_4$ such that $p^t\notin C^-_5$. The proof of this case is identical to the one for Case 1 in Lemma \ref{0116}.

{\textbf{Case 2}:} There does not exist $\tau_5 \leq t'<\sigma_4$ such that $p^t\in C^-_5$.

$p^{\tau_5}\in B_6$ implies that $p_5^{\tau_5}\geq \frac{0.55*1.55^4}{1.55^5-1}>0.39$. Let $$t_1:=\min\{t>\tau_5:p_5^t\leq 0.23 \}\text{ and }
t_2=\min\{t>t_1:p_5^t\leq 0.1\}.$$
We verify using Mathematica that while $p^t\in B_5^-$, if $t<t_1$ then $\mathbb{E}(D_t(2m(p_4-a_4p_3+0.2p_5))$ is strictly positive while if $t<t_2$ then $\mathbb{E}(D_t(2m(p_4-a_4p_3))$ is strictly positive
(in both calculations $\mathbb{I}(\z_t=0)$ is substituted by $1-p_1^t\l_1^t\sum_{i=1}^5ip_1q_{i,i}$ - see Lemma \ref{eq:sumzetas}). Hence either $t_2\geq \tau_4$ and w.s.h.p.\@ $p^t\in B_5$ for $t\in[\tau_5,\tau_4]$ or $t_2<\tau_4$. Now if $t_2<\tau_4$, from a standard martingale argument, we have that w.s.h.p.\@ 
\begin{align*}
    p_4^{t_2}-\a_4p_3^{t_2} \geq (0.39-0.23) \cdot 0.2+ 0.09\cdot 0 \geq 0.5\cdot 10^{-2}+0.025\geq  0.01p_3^{t_2}+0.25p_6^{t_1}
\end{align*}
which brings us to Case 1. Thus in Case 2, w.s.h.p.\@ $t_2\geq \tau_4'$ and $p^t\in B_5$ for $t\in[\tau_5,\tau_4]$.
\end{proof}

\subsection{The process during the interval $[\tau_4,\tau_3]$}
\begin{lemma}\label{0114}
Conditioned on $p^{\tau_4}\in B_5$ w.s.h.p.\@ \eqref{eq:app} holds for $t\in[\tau_4,\tau_3]$. In addition, w.s.h.p.\@ $p_2^{\tau_3}\geq a_2p_1^{\tau_2}$.
 \end{lemma}

\begin{proof}(Sketch)  
First we verify that while $p_3^t\leq 0.36$ and $p^t\in B_4$ we have that the following quantity (corresponding to $\mathbb{E}(D_t(2m(p_3-\a_3p_2))|\cH_t)$)  is strictly positive. 
$$(1+p_4-p_1\lambda_1\sum_{i=1}^4 ip_iq_{i+1,i+1})\l_4-[(\a_3+1)\l_3+1]p_3+\a_3(\l_2+1)p_2 -p_1\l_1(4p_3q_{4,4}+ 3p_2q_{3,3})$$
 Let $t'=\min\{t\geq \tau_4: p_3^t\geq 0.36\}.$
 Then either $t'\geq \tau_3$ or $t'< \tau_3$. In the first case, as $\sigma_3\leq \tau_3$ we have that w.s.h.p.\@ $p^t\in B_4$ for $t\in[\tau_4,\tau_3]$. As such Lemma \ref{02} implies that w.s.h.p.\@ \eqref{eq:app} holds for $t\in [\tau_4,\tau_3]$.
 
In the second case $p_3^{t'}\geq 0.36$ for some $t< \tau_3$. Using Mathematica we verify that the following hold for $t\geq \tau_4$.
\begin{itemize}
    \item[(i)] If $p_3^{t}\geq 0.35$ and $p_2^t\geq 1.55p_1^t$  then the expected change of $2m_t(p_3^t-0.36)$ is positive.
    \item[(ii)] If $p_3^{t}\geq 0.35$ and $p_2^t\leq 1.552p_1^t$ then  the expected change of $2m_t(p_2^t-\a_2p_1)$ is positive,
    \item[(iii)] If $p_3^{t}\geq 0.35$ and $p_2^t\geq 1.55p_1^t$  then \eqref{eq:app} holds.
\end{itemize}
A standard martingale argument can be used to show that the above imply that w.s.h.p.\@ for $t\in [t',\tau_3]$ we have that $p_3^t\geq 0.35$, $p_2^t>a_2p_1^t$ and \eqref{eq:app} holds. Finally \eqref{eq:app} holds while $p^t\in B_4$ and $t\geq \tau_4$ hence for $t\in [\tau_4,t']$ (see Lemma \ref{02}).
\end{proof}

\subsection{The process during the interval $[\tau_3,\tau_2]$}

\begin{lemma}\label{0113}
Conditioned on $p_2^{\tau_3}\geq a_2p_1^{\tau_2}$, w.s.h.p.\@ \eqref{eq:app} holds for $t\in [\tau_3,\tau_2]$.
\end{lemma}

\begin{proof}(Sketch)  
First we verify that while $p_2\leq 0.3$  and $p_2\geq \a_2p_1$ then the quantity 
$$(1+p_3-p_1\lambda_1\sum_{i=1}^3 ip_iq_{i+1,i+1})\l_3-[(\a_2+1)\l_2+1]p_2+\a_2(\l_1+1)p_1 -p_1\l_1(3p_2q_{3,3}+ 2p_1q_{2,2})$$
 is strictly positive. In addition
we verify using Mathematica that if $p_2^t\geq \a_2p_1^t$ then \eqref{eq:app} holds for $\tau_3\leq t$.
 
We then let $t'=\min\{t\geq \tau_3: p_2^t\geq 0.3\}.$ The above imply that 
\eqref{eq:app} holds for $\tau_3\leq t\leq t'$. Thereafter we verify using Mathematica the following.
For $t'\leq t <\tau_2$,
\begin{itemize}
    \item[(i)] if $p_3^{t}\leq 0.31$ then the expected change of $2m_t(p_2^t-0.3)$ is positive,
\item[(ii)] if $p_2^{t}\geq 0.29$ then \eqref{eq:app} holds.
\end{itemize}
The above imply that w.s.h.p.\@ \eqref{eq:app} holds for $t\in [t',\tau_2)$.
\end{proof}
\end{appendix}
\end{document}